\documentclass[11pt,leqno]{article}

 \usepackage[all]{xy}
\usepackage{amssymb,amsmath,amsthm,url,rotating}
 \usepackage{graphicx,epsfig}
 \usepackage{xcolor,graphicx}
  \usepackage{hyperref}
\newcommand{\n}{\noindent}

\newcommand{\vp}{\varepsilon}
\newcommand{\bb}[1]{\mathbb{#1}}
\newcommand{\cl}[1]{\mathcal{#1}}

\newcommand{\ovl}{\overline}

\theoremstyle{plain}
\newtheorem{thm}{Theorem}[section]
\newtheorem{lem}[thm]{Lemma}

\newtheorem{pro}[thm]{Proposition}

\newtheorem{cor}[thm]{Corollary}

\theoremstyle{definition}

\newtheorem{dfn}[thm]{Definition}

\theoremstyle{remark}
\newtheorem{rem}[thm]{Remark}

\numberwithin{equation}{section}

\def\tilde{\widetilde}

\renewcommand{\tilde}{\widetilde}

\def\R{\bb R}

\def\C{\bb C}

\def\F{\bb F}

\def\T{\bb T}

\def\d{\delta}

\def\N{\bb N}

\def\F{\bb F}

\def\T{\bb T}
\def\nl{\nolimits}
\def\d{\delta}

\def\tilde{\widetilde}

\renewcommand{\tilde}{\widetilde}

\def\R{\bb R}
\def\Z{\bb Z}

\def\C{\bb C}

\def\N{\bb N}

\def\T{\bb T}
\def\z{z}
\def\hat{\widehat}
\def\nl{\nolimits}
\begin{document}
\def\d{\delta}

  \def\tr{{\rm tr}}
 \def\y{\varphi}
\title{
Completely Sidon  sets in  discrete groups}

\author{by\\
 Gilles  Pisier\\
Texas A\&M University and   Sorbonne Universit\'e (IMJ)}

 \maketitle
 \begin{abstract}  
 A subset of a discrete group $G$
  is called
 completely Sidon if its span in $C^*(G)$ is completely isomorphic
 to the operator space version of the space $\ell_1$
(i.e. $\ell_1$ equipped with its maximal operator space structure).
We recently proved a generalization
to this context of Drury's classical union theorem for Sidon sets:
  completely Sidon sets are stable under finite unions.
 We give a different presentation of the proof
 emphasizing the ``interpolation property"
 analogous to the one Drury discovered.
 In addition we prove the analogue of the Fatou-Zygmund property:
 any bounded Hermitian function on a symmetric completely
 Sidon set $\Lambda\subset G\setminus\{1\}$ 
 extends to a positive definite function on $G$.
 In the final  section, we give a completely isomorphic characterization of the closed span
 $C_\Lambda$ of a completely  Sidon set in $C^*(G)$: the  dual (in the operator space sense) of  $C_\Lambda$
 is exact iff $\Lambda$ is  completely Sidon. In particular,
 $\Lambda$ is  completely Sidon as soon as $C_\Lambda$ is
 completely isomorphic (by an arbitrary isomorphism) to $\ell_1(\Lambda)$
 equipped with its maximal operator space structure.
 \end{abstract}  
 
 MSC Classif. 43A46, 46L06
 
 In harmonic analysis (see \cite{Kah}) a subset $\Lambda$ of an abelian
 discrete group  $G$ is called Sidon with constant $C$
 if for all finitely supported $a: \Lambda \to \C$
 we have $$\sum\nl_{n\in \Lambda} |a_n| \le C \|\sum\nl_{n\in \Lambda}  a_n \gamma_n\|_{C( \hat G)}$$
 where $\hat G$ is the dual (compact) abelian group, 
 and where $\gamma_n: \hat G \to \T$ is the character on $\hat G$
 associated to an element $n\in G$. Here 
 ${C( \hat G)}$ denotes the space of continuous functions on $\hat G$
 equipped with the usual sup-norm. For instance,
 when $G=\Z$ we may view $\hat G=\R/\Z$
 and $\gamma_n(t)=e^{2i\pi nt}$.\\
 Equivalently, if $C_\Lambda\subset {C( \hat G)}$ denotes the closed span
 of $\{\gamma_n\mid n\in \Lambda\}$ and
  $(e_n)$ denotes the canonical basis of $\ell_1(\Lambda)$
 the mapping $u: C_\Lambda \to \ell_1(\Lambda)$ defined
 by $u( \gamma_n)=e_n$  is an isomorphism
 with $\|u\|\le C$ (and trivially $\|u^{-1}\|\le 1$). 
 
 In the abelian case the subject has a long and rich history for which we refer to
 \cite{LR,MaP,Kah,GH}.
 The first period roughly 1960-1970
 was driven by a major open problem: whether the union
 of two Sidon sets is   a Sidon set. Eventually this was proved by Drury \cite{dru1}
 using a beautiful convolution device. After this achievement, it was only natural to investigate
 the non-abelian case. For that two options appear, either: \\
 1.  one
 replaces $\hat G$ by a compact non-abelian group
 and $\Lambda$ becomes a set of irreducible unitary representations
 on  the latter compact group, or:\\ 2.  one   replaces $G$ by a discrete
 non-abelian group.
 
 We will not deal with case 1; in that case
 the union problem  resisted generalization but was solved by Rider in 1975. 
 The subject suffered from the disappointing discovery that the duals of  most 
 compact Lie groups do not contain infinite Sidon sets.
 We refer the reader to our recent survey
 \cite{Prid}   for  more   on this.
 
 This paper is devoted to case 2. In this case, there were several
 attempts
to generalize the Sidon set theory notably by Picardello
and Bo\.zejko
 (see \cite{Pic,Boz}), but no analogue of Drury's union theorem was found.
 The novelty of our approach is that while these authors
 defined Sidon sets using the  Banach space structures of the
 relevant non-commutative operator algebras, we
 fully use their operator space structures. In particular,
 the Banach space $E=\ell_1(\Lambda)$ that enters  the definition of a Sidon set
 has to be considered as an operator space, given
 together with an isometric embedding $E\subset A$
 into a $C^*$-algebra $A$, or into $B(H)$ for some Hilbert space $H$.

 By definition an operator space is a subspace $E\subset A$
 (or $E\subset B(H)$).
 We may use a different $A$ and a different embedding as long
 as it induces  the same sequence of norms on all the
 spaces  $M_n(E)$  ($n\ge 1$). Of course $M_n(A)$ is equipped with
its unique $C^*$-norm, or equivalently the norm of the  $C^*$-tensor product
$M_n \otimes A$, and this induces a norm on the subspace $M_n(E)$.
The theory of operator spaces is now well developed. The main novelty
is that the bounded linear maps 
$u: E\to F$ between operator spaces
 are now replaced by
the completely bounded (in short c.b.) ones 
and the  norm $\|u\|$ is replaced by the cb-norm $\|u\|_{cb}$.
We say that $u$ is a complete isomorphism if it is
invertible and both $u$ and $u^{-1}$ are c.b. maps.
See below for background on this. We refer   to the books \cite{ER,P4}
for more information.

In the case of $\ell_1$, there is a privileged operator space structure
$\ell_1 \subset A$
 that can be conveniently described 
 using for $A$ the $C^*$-algebra $C^*(\F_\infty)$
 of the free group with countably infinitely many generators.
 Let $(U_n)$ denote the unitaries in $A$ corresponding to the free generators.
 The embedding $j: \ell_1 \subset A$
 is defined by $j(e_n)=U_n$, where $(e_n)$ is the canonical basis of  $\ell_1 $.
 Similarly, given an arbitrary set $\Lambda$ we may consider
 the group $\F_\Lambda$ freely generated by $(g_n)_{n\in \Lambda}$
 and the corresponding unitaries $(U_n)_{n\in \Lambda}$ in
 $A=C^*(\F_\Lambda)$. We then define
 $j: \ell_1(\Lambda) \subset A$ again by $j(e_n)=U_n$ for $n\in  \Lambda$.
 Following Blecher and Paulsen (see \cite[\S 3]{P4} and \cite[p. 183]{P4}), 
 we call this the   maximal operator space structure on $\ell_1(\Lambda)$.
 Unless specified otherwise, we always assume
 $\ell_1(\Lambda)$ equipped with the latter. 
 More explicitly we have for 
 any $C^*$-algebra $B$ (e.g. $B=M_N$) and any finitely supported $a: \Lambda \to B$
  \begin{equation}\label{ed11}\| \sum\nl_{\Lambda} a_t \otimes U_t\|= \sup\{ \| \sum\nl_{\Lambda} a_t \otimes z_t\| \}\end{equation}
 where the sup runs over all $H$ and all functions $z: \Lambda \to B(H)$
 such that $\sup\nl_{\Lambda} \|z_t \|\le 1$.
 \begin{rem}\label{rd13}
  By the Russo-Dye theorem,
 the supremum is unchanged if we restrict to
 $z$'s with unitary values.
 Moreover, if we wish,  we may  (after translation by $z^{-1}_s$) restrict  to $z$'s with unitary values
 and such that $z_s=1$ for a single fixed
$s\in \Lambda$.  In addition we may restrict to
 finite dimensional $H$'s if we wish (see e.g. \cite[p. 155]{P4} for details). 
  \end{rem}
  In the case
  $B=\C$, we find
  \begin{equation}\label{ed12}\| \sum\nl_{\Lambda} a_t \otimes U_t\|=   \sum\nl_{\Lambda}| a_t| .\end{equation}
 
 We   now introduce the relevant generalization of Sidon sets. 
 
 Let $A \subset B(H)$ be a $C^*$-algebra.
   If $B\subset B(K)$ is any other $C^*$-algebra (for instance $B=M_N
   =B(K)$ when $\dim(K)=N$)
   and $x\in B \otimes A$ (algebraic tensor product)
   we denote 
   by $\| x\|_{B \otimes_{ \min}  A}$  or more
   simply by $\| x\|_{\min}$ the norm of $x$ in the 
   minimal or spatial tensor product, i.e.
   we set $$ \|x\|_{\min}=\|x : K\otimes_2 H \to K\otimes_2 H\|.$$
   Moreover, we use the same definition when $A,B$ are merely operator subspaces
   of $B(H),B(K)$. It is known that $ \|x\|_{\min}$ does not depend on the
   choice of the completely isometric embeddings 
$A \subset B(H)$  and $B\subset B(K)$.

  Let $G$ be  a discrete group. 
   Let 
   $U_G: G \to B(\cl H)$ be the universal representation and 
   let  $C^*(G)\subset B(\cl H)$ denote the $C^*$-algebra generated 
   by $U_G$. 
    
    Given a subset $\Lambda\subset G$ we denote by
    $C_\Lambda \subset C^*(G)$ the operator space  defined
   by $$C_\Lambda=\ovl{\rm span}[U_G(t)\mid t\in \Lambda ].$$
   \begin{dfn}\label{d1}
   We say that $\Lambda\subset G$ is completely Sidon if
 there is $C$ such that
 for any  $N\ge 1$ and any 
 finitely supported $a: \Lambda \to M_N$ 
 $$\|\sum\nl_{t\in \Lambda} a_t \otimes U_t \|_{M_N \otimes_{\min}  C^*(\F_\Lambda)}\le C \|\sum\nl_{t\in \Lambda} a_t \otimes U_G(t) \|_{M_N \otimes_{\min}  C^*(G)}.$$
 More explicitly, this is the same as requiring 
   \begin{equation}\label{ed13}\sup\|\sum a_t \otimes u_t\|_{\min}\le C \|\sum a_t \otimes U_G(t) \|_{\min},\end{equation}
  where the sup runs over all families $(u_t)_{t\in \Lambda}$ of unitaries on
 an arbitrary Hilbert space $H$.\\
   Equivalently, 
   the linear map $u: C_\Lambda  \to \ell_1(\Lambda)$ 
   defined for $t\in   \Lambda$ by $u(U_G(t)) = U_{\F_\Lambda}(g_t)$  is c.b.
   with $\|u\|_{cb}\le C$.
     Then, since $\|u^{-1}\|_{cb}\le 1$, the space $C_\Lambda $
 is   completely isomorphic to $\ell_1(\Lambda)$ equipped with its maximal operator space structure.
 \end{dfn}

 The fundamental example is given by free sets, as follows.
 \begin{pro}\label{pd13} Let $S\subset G$ be a free set,
 and let $\Lambda$ be a translate of $S \cup \{1\}$.
 Then $\Lambda$ is completely Sidon with $C=1$.
 Conversely, any completely Sidon set with $C=1$ is of this form.
 \end{pro}
 For the proof see Proposition \ref{p02} below.
 
 We can now state our main results:\\
 1. Completely Sidon sets are stable by finite unions.\\
 2. Assume $\Lambda  $ completely Sidon, symmetric, $1\not\in \Lambda  $  
 and assume for simplicity $\Lambda  $
  without any element of order 2 (this case can also be handled), then the linear map $u: C_\Lambda \to C^*(\F_\Lambda)$
 associated to the mapping $t\mapsto g_t$  extends to
 a completely positive (in short c.p.) map
 $\check u: C^*(G) \to C^*(\F_\Lambda)$.\\
 3. If the operator space $C_\Lambda$ is completely isomorphic
 to $\ell_1(\Lambda)$ via an arbitrary linear  correspondence,
 or if the dual operator space $C^*_\Lambda$
 is exact, then $\Lambda$ is completely Sidon.
 
 Point 1 is the non-abelian version of Drury's 1970 union theorem from \cite{dru1}.
 Point  2 is  analogous to  the so-called ``Fatou-Zygmund"
 property established by Drury in 1974 (see \cite{dru2,LR}), while
 point 3  is the analogue of the 1976 Varopoulos
theorem from \cite{V}. For emphasis, we should point out 
that a surprising dichotomy stems from it: for any infinite subset $\Lambda\subset G$
the space $C^*_\Lambda$ is (roughly)   either ``very big"
 or ``very small"  in the operator space sense.
 
 Points 1 and 2 answer questions raise by Bozejko
 in \cite{Boz} (see Remark \ref{boz}).
  The proof of 
 Point 2 is similar to that of 1, but is better understood
 if one first runs through the proof of 1 as we do below.
 Moreover, the quantitative estimates we give in terms of
 the constant $C$ may be of independent interest.
  Lastly 3 is new.
 \begin{rem}\label{r22}
 We should emphasize that the theory of completely Sidon sets
 does not contain the classical case, although it is very much parallel
 to it. Indeed,
  any group $G$ that contains
 an infinite completely Sidon set must be non-amenable
 (and hence extremely non-commutative)
 because $C^*(G)$ cannot be exact.
 More precisely, if the set has at least $n$ elements
 with completely Sidon constant $C < n/2\sqrt{n-1}$
 then $C^*(G)$ is not exact
  (see \cite[p. 336]{P4}) and a fortiori $G$ is not amenable.
 However, we do not know whether such a $G$ 
 must contain a copy of $\F_\infty$ (or equivalently  $\F_2$).
 \end{rem}
 {\bf Problem: } By our main result, any finite union of translates of free sets
 is completely Sidon. Is the converse true ?
 This fundamental question is analogous to a well known open one
 for the classical Sidon sets (see \cite[p. 107]{GH}).
 
 \section{Notation and background}

  Let $E\subset B(H)$ and
$F\subset B(K)$ be operator spaces, consider a
map $u: E \to F$.
For any $n\ge 1$, let
$M_n(E)  $ be the space of $n\times
n$ matrices with entries in $E$.
We have $M_n(E) \subset M_n(B(H))$.
We equip  
$M_n(E) \subset M_n(B(H))$
with the norm induced by
$B(\ell^n_2(H)) \approx M_n(B(H))$
where $\ell^n_2(H)$ means $ {H\oplus H\oplus \cdots \oplus H}$ (${n
\ {\rm times}} )$.
We define
$u_n\colon \ M_n(E)\longrightarrow
M_n(F)$ by setting $u_n([a_{ij}])=[u(a_{ij})]$. 
A map $u\colon\ E\to F$ is called completely bounded (in short c.b.) if
$\sup\nl_{n\ge 1} \|u_n\|_{M_n(E)\to M_n(F)}<\infty.$
Let
$$\|u\|_{cb} = \sup\nl_{n\ge 1} \|u_n\|_{M_n(E)\to M_n(F)}.$$
We  denote by
$CB(E,F)$
the Banach space of all  such maps  equipped with the
c.b.\ norm.
Let $M_n(E)_+=M_n(E) \cap M_n(B(H))_+$.
We say that $u$ is completely positive (c.p. in short)
if $u_n$ is positivity preserving, i.e. $u_n(M_n(E)_+) \subset M_n(F)_+ $ for any $n$.
When $E$ is an operator system (i.e. $E$ is a unital self-adjoint linear subspace)
c.p. implies c.b. and $\|u\|_{cb}=\|u\|=\|u(1)\|$. We denote by $CP(E,F)$
the set of c.p. maps.

Let  $A,B$ be $C^*$-algebras.
We will denote by $D(A,B)$ the set of all ``decomposable'' maps $u:A\to B$, i.e. the maps that are in the linear span of $CP(A,B)$.  This means that $u\in D(A,B)$ iff
there are $u_j \in CP(A,B)\quad(j=1,2,3,4)$ such that
$$u=u_1 -u_2 +i(u_3 -u_4 ).$$
 We will repeatedly use  the nice definition of
the dec-norm of a linear map  $u: A \to B$ between $C^*$-algebras given
by  Haagerup in \cite{Haa}, as follows.
    We set
 \begin{equation}\label{d11}\| u\|_{dec}=\inf\{\max\{\| S_1\|,\| S_2\|\}\}
 \end{equation}
  where the infimum runs over all maps $S_1 ,S_2\in CP(A,B)$ such that the map
 \begin{equation}\label{d12} V: x\to \left(
\begin{matrix}
 S_1 (x) & u(x)\\
 u(x^* )^* & S_2 (x)
\end{matrix}
\right) \end{equation}
is in $CP(A,M_2 (B))$.  
This is equivalent to the simple minded choice of norm
$\| u\|=\inf \sum\nl^{4}_{1}\| u_j \|$.
When $u$ is self-adjoint (i.e. when $u(x^*)=u(x)^*$
for all $x\in A$) we have $\| u\|_{dec}=\inf\|u_1+u_2\|$
 where the infimum runs over all the possible decompositions
of $u$ as $u=u_1-u_2$ with $u_1,u_2$ c.p..

  See \cite{Haa} for the proofs of all the basic facts on decomposable maps,
  that are freely used throughout this note.
  In particular, we repeatedly use the fact
  that for any pair 
 $v_j: A_j\to B_j$ ($j=1,2$)
 of decomposable maps 
 between
 $C^*$-algebras, 
 the map $v_1\otimes v_2$ on the algebraic tensor product uniquely
 extends to a map, still denoted by $v_1\otimes v_2$, in $D ( A_1\otimes_{\max} A_2  , B_1\otimes_{\max} B_2  )$ 
 with
  \begin{equation}\label{e8}
  \| v_1\otimes v_2:   A_1\otimes_{\max} A_2\to B_1\otimes_{\max} B_2\|_{dec}   
\le \|v_1\|_{dec}\|v_2\|_{dec}.
 \end{equation}
 Moreover, if $v_1,v_2$ are completely positive (c.p. in short)
 the resulting map $v_1\otimes v_2:   A_1\otimes_{\max} A_2\to B_1\otimes_{\max} B_2$
 is  c.p..
  Here  $A_1\otimes_{\max} A_2$  stands for 
  the $C^*$-algebra obtained by completing the algebraic tensor product $A_1\otimes A_2$ with respect to the maximal $C^*$-norm (see e.g. \cite[p. 227]{P4}).
  
  We also use  from \cite{Haa} that if $B=B(H)$ or if $B$ is an injective $C^*$-algebra
  (which means the identity of $B$ factors
 through $B(H)$ via c.p. maps) then for any $C^*$-algebra $A$
 we have
 $ CB(A,B)=D(A,B)$ and for any $u\in CB(A,B)$
  \begin{equation}\label{e1} 
  \|u\|_{cb}=\|u\|_{dec}.\end{equation}  
  See \cite{ER,P4} for more background and references.

 Let $\Lambda\subset G$ be a subset of a discrete group $G$.
 
 Let $U_G$ be the universal representation of $G$, and let $C^*(G)$
 be the $C^*$-algebra generated by $U_G$.
 
 Let $\lambda_G$ be the left regular representation, and let 
   $C_\lambda^*(G)$ 
 be the $C^*$-algebra generated by $\lambda_G$. We denote  by
  $M_G$ 
 the von Neumann algebra generated by $\lambda_G$.

 The notation $(\d_s)_{s\in G}$ is used mostly for the
 canonical basis of the group algebra $\C[G]$, and sometimes (abusively)
 for that of $\ell_2(G)$.
 As usual we view $\C[G]$ as a dense $*$-subalgebra of $C^*(G)$.

 \begin{pro}\label{pfz} Assume we have an embedding $C^*(\F_\Lambda) \subset B(\cl H)$.
 The following 
 properties are all equivalent refomulations
 of  Definition \ref{d1}:
   \item {(i)} The correspondence $ t\mapsto g_t$
   from $\Lambda$ to the free generators of $\F_\Lambda$
   extends to a c.b. linear map $u: C^*(G) \to B(\cl H)$
   with $\|u\|_{cb}\le C$.
   \item {(ii)} For any Hilbert space $H$, for any bounded mapping
   ${z}: \Lambda \to B(H)$ there is  a bounded linear map
   $u_{z}: C^*(G) \to B(  H)$ with 
   $\|u_{z}\|_{cb}\le C \sup\nl_{t\in \Lambda} \|{z}(t)\|$ such that
   $u_{z}( U_G(t))= {z}(t)$ for any $t\in \Lambda$.
     \end{pro}
   \begin{proof} If $\Lambda$ is completely Sidon then clearly (i)
   holds by the injectivity of $B(\cl H)$, and conversely (i) obviously implies
   $\Lambda$   completely Sidon.\\
   By the injectivity of $B(H)$ for any ${z}$ as in (ii) there is a linear 
   $v_{z}: B(\cl H) \to B(H)$ extending the correspondence
   $U_{\F_\Lambda}(g_t) \to {z}(t)$ ($t\in \Lambda$) with
   $\|v_{z}\|_{cb} = \sup\nl_{t\in \Lambda} \|{z}(t)\|$ (this expresses the fact
   that $\{g_t\mid t\in \Lambda\}$ is completely Sidon with constant 1).
   Then the composition $u_{z}= v_{z} u$ shows that
   (i) implies (ii). The converse is obvious.
  \end{proof}
   \begin{rem}
  If $\Lambda$ is asymmetric in the sense that   $\Lambda \cap \Lambda^{-1}=\phi$,  
   we   show in Corollary \ref{da1}
  that the  correspondence $\Lambda \ni t\mapsto g_t\in \F_\Lambda$ 
 extends
  to a c.p. map $u: C^*(G) \to C^*(\F_\Lambda)$ but then
  we only obtain $\|u\|_{cb} (=\|u\|) \le O(C^4)$.
  \end{rem}

    \begin{rem}\label{boz} In \cite{Boz} Bo\.zejko
    considers the property appearing in (ii) in Proposition \ref{pfz} 
    and he calls ``w-operator Sidon" the sets with this property.
    He calls ``operator Sidon" the sets $\Lambda\subset G$ satisfying
    $\Lambda \cap \Lambda ^{-1}=\phi$ such that any
    $B(H)$-valued   bounded
    function on $\Lambda $ admits a positive definite extension
    on $G$, and proves that free sets
    (i.e. $\{g_t\mid t\in \Lambda\}$ in $\F_\Lambda$) have this property. ``Operator Sidon" is a priori
    stronger than  ``w-operator Sidon", but actually, we will show later 
    on in this paper (see Theorem \ref{tfz}) that the two properties are equivalent. Bo\.zejko also asked whether these sets
    are stable under union. We show this in Corollary \ref{druc}.
    Our results suggest to revise the terminology:
    perhaps the term ``operator Sidon"
    should be adopted instead of our ``completely Sidon".
  \end{rem}
  \begin{rem}\label{ko}  The following observation
  plays a crucial role in this paper.
  Let $\Gamma= \F_\Lambda$.
  Let $Q: C^*(\Gamma)\to M_\Gamma$ be the $*$-homomorphism
  associated to $\lambda_\Gamma$.
  Let $E\subset C^*(G)$ be an operator subspace.
  Then for  any $u\in CB( E, C^*(\Gamma))$  there is
  ${u^\dagger}\in D(C^*(G), M_\Gamma)$ with $\|{u^\dagger}\|_{dec}\le \|u\|_{cb}$
  such that ${u^\dagger}_{|E}= Qu$. Indeed,
 free groups  satisfy Kirchberg's factorization property from \cite{Kir}.
 In particular,  
by a well known construction 
  involving ultraproducts (see Th. 6.4.3 and Th. 6.2.7 in \cite{BO}), for some $H$  the map
  $Q$ factors through $B(H)$ via c.p. contractive maps
  $Q_1: B(H) \to M_\Gamma$ and $Q_2: C^*(\Gamma) \to B(H)$
  so that $Q=Q_1Q_2$. By the injectivity of $B(H)$
  the composition $Q_2 u: E \to B(H)$ admits an extension
  $\tilde{   Q_2 u}\in CB( C^*(G),B(H))$ with
  $\|  \tilde{   Q_2 u}\|_{cb} \le \|Q_2 u\|_{cb}\le \|u\|_{cb}$.
  But by \eqref{e1}  $CB( C^*(G),B(H))=D( C^*(G),B(H))$
  isometrically. Therefore $\|  \tilde{   Q_2 u}\|_{dec}\le \|u\|_{cb}$.
  The mapping ${u^\dagger}= Q_1 \tilde{   Q_2 u}$ has the announced properties.
  If we assume in addition that  $E$ is an operator system
  and that $u$ is c.p. then we find ${u^\dagger}\in CP(C^*(G), M_\Gamma)$ 
  with $\|{u^\dagger}\|=\|{u^\dagger}\|_{dec}\le \|u\|=\|u\|_{cb}$.
  
  In particular, if $\Lambda\subset G$ is a completely Sidon set with constant $C$, 
   let $E\subset C^*(G)$ be the span of $\{U_G(t)\mid t\in \Lambda\}$. We may apply the preceding observation  to the linear mapping 
  $u$ defined by $u(U_G(t)) = U_\Gamma(t)$ ($t\in \Lambda$).
  We find ${\cl U}\in D(C^*(G), M_\Gamma)$ such that 
  ${\cl U}(U_G(t)) = \lambda_\Gamma(g_t)$  for all $t\in \Lambda$ with $\|{\cl U}\|_{dec}\le \|u\|_{cb}\le C$.
  We will show below (see Corollary \ref{cc1}) that conversely the existence of such a ${\cl U}$ implies
  that $\Lambda\subset G$ is   completely Sidon.
 \end{rem}
 
  \begin{rem}\label{dye}  
   Let $A$ be a unital $C^*$-algebra. By
    \cite{KaPe}
      any $a\in A$ with $\|a\|<1-2/n$
 can be written as an average of $n$ unitaries in $A$.
\end{rem}
 
 \section{Operator valued harmonic analysis}
 Let $G$ be a discrete group. 
 Let 
 $\varphi: G \to A$ be a function with values in a $C^*$-algebra.
 Let $u_\varphi: \C[G] \to A$ be the linear map extending $\varphi$.
 We denote respectively by
 $$B(G,A), \quad CP(G,A),\quad CB(G,A), \quad D(G,A)$$
 the set of those $\varphi$ such that $u_\varphi$ extends to
 a   map $u_\varphi: C^*(G) \to A$ respectively in
  $$B(C^*(G),A), \quad CP(C^*(G),A), \quad CB(C^*(G),A), \quad D(C^*(G),A)$$
  and we set
 \begin{equation}\label{97}
  \| \varphi \|_{B(G,A)}= \|  u_\varphi\|  , \quad 
  \| \varphi \|_{CB(G,A)}= \|  u_\varphi\|_{cb} , \quad 
  \| \varphi \|_{D(G,A)}= \|  u_\varphi\|_{dec} .
   \end{equation}

 By \eqref{e1}, when
 $A=B(H)$ or when $A$ is injective  then $ {CB(G,A)}={D(G,A)}$
 and 
 $\| \varphi \|_{CB(G,A)}=\| \varphi \|_{D(G,A)}$, but in general we only have
 $ {D(G,A)}\subset {CB(G,A)} $
 with 
 $\| \varphi \|_{CB(G,A)}\le \| \varphi \|_{D(G,A)}$ and
 the inclusion is strict.
 
 When $A=\C$ we
 have $B(G,\C)=CB(G,\C)=D(G,\C)$ isometrically and we recover the  non-commutative analogue of the classical ``Fourier-Stieltjes
algebra"  $B(G)$ (see e.g. \cite{Ey}  or \cite[p. 3]{FiPi}), which can be identified isometrically with  $C^*(G)^*$: we have
$\varphi\in B(G)=B(G,\C)$ iff
there is a unitary representation $\pi: G \to B(H)$
and vectors $\xi,\eta\in H$ such that 
$\varphi(.) =\langle\eta, \pi(. )\xi\rangle$ and
$\|\varphi\|_{B(G)}=\inf\{\|\eta\| \|\xi\|\}$ where the infimum (actually the minimum is attained) runs over all
possible such representations
of $\varphi$.

By the factorization of c.b. maps (see e.g. \cite{ER,Pa2})
the case when $A=B(H)$ is entirely analogous:\\
in that case   $\varphi\in CB(G,B(H)) $ iff
there are $\hat H$, a unitary representation $\pi: G \to B(\hat H)$
and operators $\xi,\eta\in B(H,\hat H)$ such that 
$\varphi(.) = \eta^* \pi(. )\xi $ and
\begin{equation}\label{e2}
\|\varphi\|_{CB(G,A)}=\inf\{\|\eta\| \|\xi\|\}\end{equation}
 where the infimum (actually a minimum) runs over all
possible such representations
of $\varphi$.

With this notation we can immediately reformulate Proposition \ref{pfz}
like this:
\begin{pro}\label{pd1} A subset $\Lambda \subset G$ is a completely Sidon set with constant $C$
iff for any $H$ and any    $z: \Lambda\to B(H)$
such that $\sup\nl_\Lambda\|z\|\le 1$ there is
$\varphi \in CB(G,B(H))$ with $\| \varphi  \|_{ CB(G,B(H))}\le C$
such that $\varphi_{|\Lambda}=z$. Moreover, for the latter
to hold it suffices that it holds for any finite dimensional
$H$.
\end{pro}
The next lemma is a simple refinement of the last statement. The proof
is based on a specific ``extremal" property of the norm in \eqref{ed11}.
\begin{lem}\label{ld1} Let $0<\vp<1$. Let $H$ be a Hilbert space and $c> 0$ a constant.
 Assume that  for any    $z: \Lambda\to B(H)$
 with $\sup\nl_\Lambda\|z\|\le 1$
 there is
$\varphi_0 \in CB(G,B(H))$ with $\| \varphi_0  \|_{ CB(G,B(H))}\le c$
such that $\sup_{\Lambda}\| z-\varphi_0\|\le \vp $.
Then  
for any  $z: \Lambda\to B(H)$
 with $\sup\nl_\Lambda\|z\|\le 1$
  there is
$\varphi \in CB(G,B(H))$ such that $\varphi_{|\Lambda}=z$ with $\| \varphi  \|_{ CB(G,B(H))}\le c/(1-\vp)$.
\end{lem}
\begin{proof} Applying the assumption to the function $(\varphi_0 -z)/\vp$
we find $\varphi_1\in CB(G,B(H))$ with $\| \varphi_1  \|_{ CB(G,B(H))}\le c$
such that $\sup_{\Lambda} \|z-\varphi_0-\vp \varphi_1\|\le \vp^2$.
Repeating this step, we obtain  $\varphi_j$ with $\| \varphi_j  \|_{ CB(G,B(H))}\le c$
such that $\sup_{\Lambda} \|z-\varphi_0-\cdots-\vp^j \varphi_j\|\le \vp^{j+1}$.
Then $\varphi=\sum\nl_0^\infty \vp^j\varphi_j$ gives us the desired function.
\end{proof}
\begin{rem}[On completely positive definite functions]\label{cpd}
We will say (following \cite{Pa2}) that $\varphi: G \to A$ is completely positive definite
 if for any finite subset $\{t_1,...,t_n\}\subset G$
 we have $[ \varphi(t_i^{-1}t_j)] \in M_n(A)_+$.
By classical results (due to Naimark,  see \cite[p. 51]{Pa2})
  $\varphi\in CP(G,A)$ iff $\varphi$ is completely positive definite.
  Assuming $A\subset B(H)$ $\varphi$ is completely positive definite
  iff there are $\hat H$,    $\pi: G \to B(\hat H)$
and   $\xi \in B(H,\hat H)$ such that 
$\varphi(.) =\xi^*  \pi(. )\xi $. When $A=B(H)$,
 by a polarization argument \eqref{e2} shows
that any $\varphi\in CB(G,A)$ can be written
as  a linear combination $\varphi=\varphi_1 -\varphi_2 +i(\varphi_3 -\varphi_4 )$
with $\varphi_j\in CP(G,A)$ for all $j=1,...,4$.
\end{rem}
  
 The spaces $CB(G,A)$ and  $D(G,A)$ can also be viewed
 as spaces of multipliers.
 To  any $\varphi: G \to A$
 we associate a ``multiplier"
 $M_\varphi: \C[G] \to C^*(G) \otimes_{\min} A$  that takes
 $t\in G$ to $U_G(t) \otimes \varphi(t)$.
 
 \begin{pro}\label{pd16}   The multiplier $M_\varphi$
 extends to a c.b. (resp. decomposable) map
 from $C^*(G)$ to $C^*(G) \otimes_{\min} A$ (resp. $C^*(G) \otimes_{\max} A$)
 iff $\varphi\in CB(G,A)$ (resp. $\varphi\in D(G,A)$), and we have
$$\|\varphi\|_{CB(G,A)}=\|  M_\varphi:  C^*(G)   \to C^*(G) \otimes_{\min} A \|_{cb} $$
$$\|\varphi\|_{D(G,A)}=\|  M_\varphi:  C^*(G)   \to C^*(G) \otimes_{\max} A \|_{dec}=\|  M_\varphi:  C^*(G)   \to C^*(G) \otimes_{\min} A \|_{dec}.$$
Moreover, $\varphi\in CP(G,A)$
iff $M_\varphi$
 extends to a c.p. map
 from $C^*(G)$ to $C^*(G) \otimes_{\max} A$,
 or equivalently a c.p. map
 from $C^*(G)$ to $C^*(G) \otimes_{\min} A$.
 \end{pro}
  \begin{proof}
  Let $\pi_1: G \to \C=B(\C)$ be the trivial representation and let
  $u_1: C^*(G) \to \C=B(\C)$ be the associate $*$-homomorphism.
   Note that $u_\varphi= (u_1 \otimes Id_A) M_\varphi$.
   This shows that if $M_\varphi$ is either c.b., c.p. or decomposable
   with values in $C^*(G) \otimes_{\min} A$
   then the same is true for $u_\varphi$.
   Conversely,  if $\|u_\varphi\|_{cb}\le 1$  then
    $\|Id_{C^*(G)}\otimes  u_\varphi :C^*(G) \otimes_{\min} C^*(G)\to C^*(G) \otimes_{\min} A\| _{cb}\le 1$.
    Let $J_{\min} : C^*(G) \to C^*(G) \otimes_{\min} C^*(G)$ be the diagonal embedding
    taking $t\in G$ to $(t,t)\in G \times G$ (corresponding to $U_G \simeq U_G \otimes U_G$ as representations on $G$).
    Then $M_\varphi=(Id_{C^*(G)}\otimes  u_\varphi )J_{\min} $
   and hence  $\|M_\varphi :C^*(G)  \to C^*(G) \otimes_{\min} A\| _{cb}\le 1$.
Similarly,  $u_\varphi$ c.p. implies that
$M_\varphi :C^*(G)  \to C^*(G) \otimes_{\min} A$ is c.p..\\
Assume $\|u_\varphi\|_{dec}\le 1$.
Then by 
\eqref{e8}
$\|Id_{C^*(G)}\otimes  u_\varphi :C^*(G) \otimes_{\max} C^*(G)\to C^*(G) \otimes_{\max} A\| _{dec}\le 1$.
 Let $J_{\max} : C^*(G) \to C^*(G) \otimes_{\max} C^*(G)$ be the analogous diagonal embedding so that
 $M_\varphi=(Id_{C^*(G)}\otimes  u_\varphi )J_{\max} $.
 It follows that  $\|M_\varphi :C^*(G)  \to C^*(G) \otimes_{\max} A\| _{dec}\le 1$.
 A fortiori,  
 composing
 with the $*$-homomorphism $ C^*(G) \otimes_{\max} A  \to C^*(G) \otimes_{\min} A$ we have
 $\|M_\varphi :C^*(G)  \to C^*(G) \otimes_{\min} A\| _{dec}\le 1$.
 \end{proof}
  \begin{rem}\label{pr11} By a classical result (see \cite{Ey}) $B(G)$
  (which is isometrically the same as $CB(G,\C)$ or $D(G,\C)$)
  is a Banach algebra for the pointwise product. In the
    operator valued case there are \emph{two distinct analogues} of this fact,  as follows.
  Let $A_j$ be $C^*$-algebras ($j=1,2$).
  Let $\varphi_j\in CB(G , A_j)$ (resp.  $\varphi_j\in D(G , A_j)$).
  Then the function $\varphi_1 \otimes \varphi_2: G \to A_1\otimes A_2$
  is in $CB(G, A_1\otimes_{\min} A_2)$ (resp. $D(G, A_1\otimes_{\max} A_2)$
  with norm  
  $$\| \varphi_1\otimes \varphi_2\|_{CB(G , A_1\otimes_{\min} A_2)}   \le \| \varphi_1\|_{CB(G , A_1)} \| \varphi_2\|_{CB(G , A_2)}
  $$
  $$ \text{ (resp.  }
  \| \varphi_1\otimes \varphi_2\|_{D(G , A_1\otimes_{\max} A_2)}   \le \| \varphi_1\|_{D(G , A_1)} \| \varphi_2\|_{D(G , A_2)}.)$$
  To check this it suffices to observe
  that $u_{\varphi_1\otimes \varphi_2}=(u_{\varphi_1 }\otimes u_{  \varphi_2}) J_{\min}$
  (resp. $u_{\varphi_1\otimes \varphi_2}=(u_{\varphi_1 }\otimes u_{  \varphi_2}) J_{\max}$).
  Moreover,  in both cases $\varphi_1\otimes \varphi_2$ is completely positive definite
  if  each $\varphi_1, \varphi_2$ is so.
  \end{rem}
 We now investigate the converse direction: how to
 obtain a multiplier from a linear mapping.
 
 \begin{rem}\label{pr1} We have an embedding $G \to  G\times G$ as a diagonal subgroup
  $\Delta_G\subset G\times G$.
  In that case it is well known (see e.g. \cite[p. 154]{P4})
  that we have a c.p. projection  from $C^*(G\times G)$
  onto the closed span of the subgroup $\Delta_G$ in $C^*(G\times G)$.
  It follows that the  map
  $P_{\max} : C^*(G) \otimes_{\max}C^*(G) \to C^*(G) $
   defined by $P_{\max}  (U_G(s) \otimes U_G(t))=U_G(t)$ if $s=t$
  and $=0$ otherwise is a unital  c.p. map such that
  $\| P_{\max}\|=\| P_{\max}\|_{dec}=1$.
 Moreover, obviously $P_{\max} J_{\max}=Id_{C^*(G)}$.
 Therefore $J_{\max}  P_{\max}$ is a unital c.p. projection (a conditional expectation)
 from $ C^*(G) \otimes_{\max}C^*(G)$ to  $ J_{\max} (C^*(G) )\simeq C^*(\Delta_G)$.
 \end{rem}
 For any $t\in G$ we denote by $f_t^G\in M_G^* $
 the functional defined by
 $$f_t^G(x)=\langle  \d_t, x \d_e\rangle.$$
Note that $(f_t^G)$ is biorthogonal to $(\lambda_G(t))$.

 The next result, essentially
 from \cite[p.150]{P4}, is a   refinement
 of Remark \ref{pr1}, that
  illustrates the usefulness of the Fell absorption 
 principle. The latter says that for any unitary representation
 $\pi$ on $G$ the representation $\lambda_G\otimes \pi$
 is unitarily equivalent to $\lambda_G \otimes I$ (see e.g. \cite[p. 149]{P4}).

\begin{thm}\label{ppp} We have an isometric ($C^*$-algebraic) embedding
$$J_G\colon \ C^*(G) \subset M_G \otimes_{\rm max}
M_G$$
taking $U_G(t)$ to $\lambda_G(t) \otimes \lambda_G(t)$ $(t\in G)$,
and  
  a completely contractive c.p. 
  mapping $$P_G\colon \ M_G \otimes_{\rm max}
M_G\to C^*(G)  $$
such that $$Id_{C^*(G) }=P_GJ_G.$$ 
Moreover,  
$\forall a,b\in M_G$, $a=\sum\nl_{t\in G} a(t) \lambda_G(t)$,
 $b=\sum\nl_{t\in G} b(t) \lambda_G(t)$ we have (absolutely convergent series)
 $$P_G(a\otimes b)=  \sum\nl_{t\in G} a(t) b(t) U_G(t).$$
\end{thm} We illustrate this by the following diagram: we have  $J_G=(\Upsilon_G \otimes  \Upsilon_G)J_{\max}$:
$$\xymatrix{
C^*(G) \otimes_{\rm max} 
C^*(G)\ar[r]^{\quad \Upsilon_G \otimes  \Upsilon_G} 
& M_G \otimes_{\rm max} M_G \ar[d]^{P_G} \\
C^*(G)\ar[u]^{J_{\max}} \ar[r]^{ Id_{C^*(G)}} & C^*(G) }$$
where  $J_{\max}: C^*(G) \to C^*(G) \otimes_{\max} C^*(G)$ (as before)
and $\Upsilon_G: C^*(G) \to M_G$  are
  the 
  $*$-homomorphisms determined by
  \begin{equation}\label{da06}\forall t\in G\quad J_{\max} ( U_G(t) )=U_G(t) \otimes U_G(t) \text{  and  } \Upsilon_G( U_G(t) )= \lambda_G(t).
  \end{equation}
  \begin{proof}
Let $x\in M_G \otimes M_G$ (algebraic tensor product).
For $s,t\in G$ let
  $x(s,t) = (f_s^G \otimes f_t^G)(x)$.
Note that $(a\otimes b)(s,t)=  f_s^G( a)f_t^G( b)$
and $(\sum\nl_s |f_s^G( y)|^2)^{1/2} \le \|y\|_{M_G}$ for any $y\in M_G$.
Therefore  $\sum\nl_t |(a\otimes b)(t,t)| \le \|a\|_{M_G} \|b\|_{M_G}$.
This shows that $\sum\nl_t |x(t,t)|<\infty$ for any $x\in M_G \otimes M_G$.\\
Note for future reference that for any $b\in M_G$
\begin{equation}\label{06}
P_G(\lambda_G(t) \otimes b)= f_t^G(b) U_G(t).
\end{equation}
We will
show the following claim:
\begin{equation}\label{(6)}\left\|\sum\nl_t x(t,t) U_G(t) \right\|_{C^*(G)} 
\le \left\| x\right\|_{M_G \otimes_{\max} M_G} .\end{equation}
Then we set $P_G(x)= \sum\nl_t x(t,t) U_G(t)$.
This implies the result.
Indeed, in the converse direction we have obviously
$$\left\|\sum x(t,t) \lambda_G(t) \otimes \lambda_G(t) \right\|_{\rm max}
\le \left\|\sum x(t,t) U_G(t) \right\|,$$
and hence \eqref{(6)} implies at the same time that $J_G$ defines an isometric
$*$-homomorphism and that the natural (``diagonal") projection   onto
$ J_G({C^*(G)}) $ is a
contractive map (actually a conditional expectation).  
The proof of the claim will actually show that $P_G$ is c.p..
We now prove this claim. Let $\pi\colon \ G\to B(H)$ be
a unitary representation of $G$. 
As usual we denote by $\rho_G$ the right regular representation
taking any $t\in G$ to the unitary of right translation by $t^{-1}$.
We introduce a pair of commuting
representations $(\pi_1,\pi_2)$ on $ \ell_2(G)\otimes_2 H$ as follows:
$$\pi_1(\lambda_G(t)) = \lambda_G(t) \otimes \pi(t)\quad \hbox{and}\quad
\pi_2(\lambda_G(t)) = \rho_G(t)\otimes I.$$
Note that both $\pi_1$ and $\pi_2$ extend to normal isometric representations on
$M_G$. For $\pi_1$  this follows from the Fell absorption
principle. For $\pi_2$, it follows from the fact that $\rho_G
\simeq~\lambda_G$ (indeed if $W\colon \ \ell_2(G)\to \ell_2(G)$ is the
unitary taking $\delta_t$ to $\delta_{t^{-1}}$, then $W^*
\lambda_G(\cdot)W = \rho_G(\cdot)$).\\
We denote by
  $\pi_1 .\pi_2 : M_G \otimes M_G \to { B(\ell_2(G)\otimes_2 H)}$
  the linear map (actually a $*$-homomorphism) defined on finite sums of rank 1 tensors by
  $(\pi_1 .\pi_2) (\sum a_j\otimes b_j)=\sum \pi_1(a_j) \pi_2(b_j)$.

\n Since $\pi_1$ and $\pi_2$ have commuting ranges, we have
\begin{equation}\label{(770)}  \left\| (\pi_1 .\pi_2)(x)\right\|_{ B(\ell_2(G)\otimes_2 H)}
\le \left\| x\right\|_{M_G \otimes_{\max} M_G},\end{equation}
hence compressing  the left-hand side to $K=\delta_e \otimes H \subset
\ell_2(G) \otimes_2 H$, we obtain (note that $\langle  \delta_e,
\lambda_G(s) \rho_G(t)\delta_e \rangle =1$
if $s=t$ and zero otherwise) 
$$\sum\nl_t x(t,t) \pi(t)= P_K(\pi_1 .\pi_2)(x)_{|K}$$ and hence
\begin{equation}\label{(77)}\left\|\sum\nl_t x(t,t) \pi(t) \right\|_{B(H)}
 \le
  \left\| x\right\|_{M_G \otimes_{\max} M_G}.\end{equation}
Finally, taking the supremum over $\pi$, we obtain the announced claim
\eqref{(6)}.
This argument shows that
$P_G$ is c.p. and  $\|P_G\|_{cb}\le 1$.
\end{proof}
\begin{rem}\label{07}
Let us denote $\bar A$ the complex conjugate of a $C^*$-algebra $A$,
i.e. $A$ equipped with scalar multiplication defined 
for $\alpha\in \C,a \in A$ by: $\alpha  \bar a=\ovl{\bar{ \alpha} a}$,
where $\bar a$ denotes $a$ viewed as an element of $\bar A$. 
Note that  the correspondence 
$U_G(t) \mapsto \ovl{U_G(t)}$ (resp. $\lambda_G(t) \mapsto \ovl{\lambda_G(t)}$)
extends to a $\C$-linear isomorphism from $C^*(G)$ to $\ovl{C^*(G)}$
(resp.  from $M_G$ to $\ovl{M_G}$).
Therefore the following variant of  Theorem \ref{ppp}
also holds: {\it There is an embedding $j_G:  C^*(G) \to \ovl{M_G} \otimes_{\max} M_G$ 
that takes $t\in G$ to $\ovl{\lambda_G(t)}  \otimes {\lambda_G(t)}$
and a contractive c.p. map $p_G :\ovl{M_G} \otimes_{\max} M_G \to C^*(G) $
such that $p_G(   \ovl{\lambda_G(t)} \otimes  {\lambda_G(s)} )=0$ for $t\not=s$
satisfying $p_Gj_G=Id_{ C^*(G)}$.}
\end{rem}
\begin{cor}\label{cc1} Let $\Gamma=\F_{\Lambda}$. A subset 
$\Lambda\subset G$ is   completely Sidon  iff
there is  a ${\cl U}$ in $D(C^*(G), M_\Gamma)$  
such that
${\cl U}(U_G(t)) = \lambda_\Gamma(g_t)$  for all $t\in \Lambda$.
In that case, the Sidon constant is at most   $ \|{\cl U}\|_{dec}^2.$
 \end{cor}
  \begin{proof} Assume there is such a $\cl U$. 
  Let $ J_{\max} $ be as in \eqref{da06}.
 Consider the mapping 
  $$J= P_\Gamma ({\cl U}\otimes {\cl U}) J_{\max} .$$
  Clearly $J(U_G(t))=U_\Gamma(g_t)$ for all $t\in \Lambda$.
  By Theorem \ref{ppp} and \eqref{e8} we have
  $$\|J: C^*(G)\to C^*(\Gamma)\|_{dec} \le \|{\cl U}\|_{dec}^2.$$
  A fortiori $\|J\|_{cb} \le \|{\cl U}\|_{dec}^2$. By Proposition \ref{pfz} $\Lambda$ is completely Sidon with constant
  $ \|{\cl U}\|_{dec}^2.$
  \\ For the converse, see Remark \ref{ko}.
  \end{proof}
  
  \begin{pro}\label{pd15}   Let $u\in D(C^*(G), M_G \otimes_{\max} A)$ with $\|u\|_{dec}\le 1$.
 We define $\varphi_u : G \to A$ by
 $$\varphi_u(t)= (f_t^G \otimes Id_A)u(U_G(t)).$$
 Then $\|   \varphi_u \|_{D(G,A)}\le 1$. If $u$ is c.p. then $\varphi_u\in CP(G,A)$. \\
   Moreover, if  there is $\varphi : G \to A$
   such that $u( U_G(t)) =\lambda_G(t)\otimes \varphi(t)$ for all $t\in G$,
  then $ \varphi_u =\varphi$.
 \end{pro}
  \begin{proof} 
  Let $u_\lambda: C^*(G)\to M_G$ be the $*$-homomorphism
  taking $t\in G$ to $\lambda_G(t)$. By \eqref{e8}
  $$\| u_\lambda\otimes u:  C^*(G) \otimes_{\max}C^*(G)  \to M_G \otimes_{\max}M_G \otimes_{\max} A \|_{dec}\le 1.$$
  Let
  $v= (P_G\otimes Id_A)(u_\lambda\otimes u) J_G $. Then
  $v( U_G(t))=U_G(t)\otimes \varphi_u(t)$ by \eqref{06} and $\|v\|_{dec} \le 1$.
  With $u_1$ associated as above to the trivial representation    
  $u_1 v( U_G(t))=  \varphi_u(t)$ and hence
  $\|\varphi_u \|_{D(G,A)}=\|u_1 v\|_{dec}\le 1$.
  If $u$ is c.p. so is $u_1 v$ and $\varphi_u\in CP(G,A)$.
  The last assertion is immediate.
       \end{proof}
        
        Let $\Gamma$ be another discrete group.
Let $T\in D( C^*(G), M_\Gamma)$.
Let $T({\gamma},s)$ be the associated ``matrix" 
defined by   \begin{equation}\label{se1a} T({\gamma},s)=f_{\gamma}^{\Gamma}( T(U_G(s)) )\end{equation}
and determined by the identity
  $T (s)=\sum\nl_{\gamma} T({\gamma},s) \lambda_\Gamma({\gamma}),$
  where the convergence is in $L_2(\tau_\Gamma)$.
  Note 
  \begin{equation}\label{se1} \sup\nl_{s\in G} (\sum\nl_ {{\gamma}\in \Gamma} | T({\gamma},s) |^2)^{1/2}\le \|T\|.\end{equation}
  
  We will use the following special case of Proposition \ref{pd15}.
  
    \begin{lem}\label{s4} 
    Let $v \in D( C^*(G),C^*(G))$. Let $T_v=u_\lambda v \in D(C^*(G),M_G)$ and let
    $T_v(t,s)$ be the associated matrix      as in \eqref{se1a}.
    Let  $v^{\bullet}: G \to\C$ be the function defined by
     $$v^{\bullet}(t)= T_v(t,t).$$
    Then $v^{\bullet}\in B(G)$ and
$$\| v^{\bullet} \|_{B(G)}= \| v^{\bullet}\|_{CB(G,\C)} = \|v^{\bullet} \|_{D(G)} \le \|v\|_{dec}.$$
        \end{lem}
        
\begin{proof}  We apply Proposition \ref{pd15}
with $A=\C$ and $u= u_\lambda v$. Then $\varphi_u=v^{\bullet}$ and
$\|v^{\bullet} \|_{D(G)} \le \|u_\lambda v\|_{dec} \le \|v\|_{dec}$. The
isometric identities
 $CB(G,\C)=D(G,\C)=B(G,\C)$ give the rest.
    \end{proof}
    
     \begin{rem}\label{s5} Let $A$ be a $C^*$-algebra, let  $v \in D( C^*(G),C^*(G)\otimes_{\max} A)$ and let $u= (u_\lambda  \otimes Id_A)v$. We will again denote
    $ v^{\bullet}=\varphi_u$ where $\varphi_u: G \to A$
 is the function defined in Proposition \ref{pd15}.
 We then have $\|v^{\bullet}\|_{D(G,A)} \le  \|v\|_{dec}$. Moreover,
 $ v^{\bullet}\in P(G,A)$ if $v$ is c.p.. 
 \end{rem}
 More generally we will use the following variant of Lemma \ref{s4}.
 \begin{lem}\label{s5b} 
  Let $\Gamma$ be another discrete group.
  Assume that there is a group morphism $q: \Gamma \to G$.
 Let  $T \in D( C^*(G),C^*(\Gamma) )$ 
 such that there is a scalar matrix $[T(\gamma,s)]$ ($\gamma\in \Gamma, s\in G$)
 satisfying
 $$\forall s\in G\quad \sum\nl_{\gamma\in \Gamma} |T(\gamma,s)   |<\infty \text{  and  }
 T ( U_G(s) ) = \sum\nl_{\gamma\in \Gamma} T(\gamma,s) U_\Gamma(\gamma).$$
 Let $\Theta : C^*(G)\to C^*(\Gamma) $ be defined by
 $$\Theta( U_G(s) ) = \sum\nl_{\gamma\in \Gamma, q(\gamma)=s} T(\gamma,s) U_\Gamma(\gamma).$$
 Then $\Theta \in D( C^*(G),C^*(\Gamma) )$ with $\|\Theta\|_{dec}
 \le \|T\|_{dec}$.
  Moreover,
 $  \Theta$ is c.p. if $T$ is c.p.. 
 \end{lem}
        
         \begin{proof}  
         Let $\hat q: C^*(\Gamma)\to C^*(G) $ denote the $*$-homomorphism 
         associated to  $q: \Gamma \to G$.
          Let $T_1= J_\Gamma T: C^*(G) \to  C^*(\Gamma)  \otimes_{\max} C^*(\Gamma)$, and $v= (\hat q \otimes Id_{C^*(\Gamma)}) T_1$.
 Clearly $v\in D(  C^*(G) , C^*(G)    \otimes_{\max} C^*(\Gamma)  )$
 with   $\|v\|_{dec}\le \|T\|_{dec}$.
 Let $\Psi = v^\bullet: G \to C^*(\Gamma)$ be the function defined
 in Remark \ref{s5} and let $  \Theta: C^*(G)\to C^*(\Gamma) $ be the linear map associated to  $\Psi $. Then the latter
 implies $\| \Theta\|_{dec}=\| \Psi\|_{D(G,C^*(\Gamma))}  \le \|T\|_{dec}$.
   \end{proof}
        
 \begin{rem} Let $J_{\max} : C^*(G) \to C^*(G)\otimes_{\max} C^*(G)$
  and $P_{\max} : C^*(G)\otimes_{\max} C^*(G)\to C^*(G) $ be as before.
 Let $\chi_G=\Upsilon_G \otimes \Upsilon_G : C^*(G)\otimes_{\max} C^*(G)  \to M_G\otimes_{\max} M_G$
 with $ \Upsilon_G $ as in \eqref{da06}.
    Then, recapitulating,  we have $$ P_{\max} = P_G \chi_G,\quad J_G= \chi_G J_{\max}   \ \text{    and    } \  P_{\max} J_{\max}=P_G J_G=Id_{C^*(G) }.$$
     \end{rem}

 \section{Interpolation}
 
 We start by an interpolation theorem that can be viewed as a non-commutative  
  Drury trick.
 
 \begin{thm}\label{dru} Let $\Lambda\subset G$ be a completely Sidon set with constant $C$. Let $w(\vp)=C^2/\vp$ for $\vp>0$.
  For any $0\le \vp\le 1$  there is 
  a function   $\psi_\vp\in B(G)$
 with 
 $\| \psi_\vp \|_{B(G)}\le w(\vp)$ such that
$\psi_\vp ( s)=1$ for any $s\in \Lambda$ and
$ |   \psi_\vp ( s)| \le   C^2\vp   $ for any $s\not \in \Lambda$.\\
More generally, for any  $0\le \vp\le 1$ and any function $z: \Lambda \to A$
with values in a unital $C^*$-algebra $A$ with $\sup\nl_\Lambda\| z \|< 1$
there is $\psi_{\vp,z}\in D(G,A)$ with 
 $\| \psi_{\vp,z} \|_{D(G,A)}\le w(\vp)$ such that
$${\psi_{\vp,z} }_{|\Lambda}= z  \text{  and  } 
 \sup\nl_{G\setminus\Lambda} \|   \psi_{\vp,z} \|_A  \le  C^2 \vp   .$$ 
 \end{thm}
 \begin{proof}[Outline of proof]
 The first step is the special case when  $G=\F_\Lambda$
 for the set $\tilde\Lambda\subset \F_\Lambda$
 formed of the free generators indexed by  $\Lambda$ (see Lemma \ref{c99}).
 The second step (Lemma \ref{ld13a})
 establishes a strong link between the set $\Lambda$ and the set
 $\tilde\Lambda$.
 We will then complete the proof (after Remark \ref{ld13aa})
 by transplanting the case of  $\tilde\Lambda\subset \F_\Lambda$ to 
 that of  $ \Lambda\subset G$.
 \end{proof}
  \begin{rem} Note that when $A=B(H)$, if we settle for a weaker estimate, the first part implies
  the second one.
  Indeed, let $z: \Lambda \to B(H)$ with $\sup\nl_\Lambda\|z\|\le 1$
  and let $\varphi\in CB(G,B(H))$ with $\|  \varphi\|_{CB(G,B(H))}\le C$ extending $\z$
  as in Proposition \ref{pd1}.
  Then the function ${\psi}_{\vp,z} = \varphi \psi_{\vp}$
  satisfies ${\psi_{\vp,z} }_{|\Lambda}= z$,   
  $\|{\psi}_{\vp,z} \|_{D(G,B(H))} \le C w(\vp)$
  and   $
 \sup\nl_{G\setminus\Lambda} \|   {\psi}_{\vp,z} \|_{B(H)} \le  C^3 \vp   .$
   \end{rem}
 Using this statement, the following is immediate
 by well known arguments.
  \begin{cor}\label{druc}  
  The union of two completely Sidon sets  
  is completely Sidon.
 \end{cor}
 \begin{proof} Fix $0<\vp<1$.
 Let $\Lambda_1,\Lambda_2$ be  completely Sidon sets in $G$
 with respective constants $C_1,C_2$ and let $\Lambda=\Lambda_1\cup\Lambda_2$.
 We may and do assume $\Lambda_1,\Lambda_2$ disjoint.
  Let $z: \Lambda \to B(H)$
   with $\sup_\Lambda\|z\|\le 1$. 
   By Theorem \ref{dru}  (recalling \eqref{e1}) there are  
   $\varphi_j\in CB(G, B(H))$  
   with $\|\varphi_j  \|_{CB(G, B(H))} \le C_j^4/\vp $
   such that 
   $\varphi_j=z$ on $\Lambda_j$ and $\sup_{G\setminus\Lambda_j}\| \varphi_j \| \le \vp$ for both $j=1,2$.
   Then  $\varphi =\varphi_1 +\varphi_2$
   satisfies
   $\|\varphi\|_{CB(G, B(H))}  \le (C_1^4+C_2^4)/\vp$ 
and $\sup_{\Lambda} \|\varphi-z\|\le \vp$. By Proposition
   \ref{pd1} this shows that $\Lambda$
   is completely Sidon with constant $(C_1^4+C_2^4)/(\vp(1-\vp))$.
    \end{proof}
    \begin{rem}[Can the estimates be improved ?]\label{mela} 
    Actually as the proof below shows, we can use
    for $w$ any function $w$ such that Theorem \ref{dru} holds
    when $\Lambda=\{g_t\mid t\in \Lambda\}\subset \F_\Lambda$.
    Given the spectrum of the Haagerup multiplier $h_\vp(t)=\vp^{|t|}$
    appearing below    (that generalizes Riesz products to the non-commutative case) 
    we may apply an argument due to  M\'ela \cite[Lemme 3]{Mel}
    for which we refer for more details to \cite[Remark 1.16]{Pi3}
    that implies that 
Theorem \ref{dru} holds for  a better $w$, namely
for  $w(\vp) =  C^2 c_1\log( 2/\vp) $ for some numerical constant $c_1>0$ (instead of $w(\vp) = C^2/\vp$). 
In the preceding corollary, assuming $C=\max\{C_1,C_2\}$ large,
this leads to $\Lambda=\Lambda_1\cup\Lambda_2$ completely Sidon
with  a constant  $C(\Lambda)=O( C^2 \log C)$. 
This   same estimate has been known for  Sidon sets
since M\'ela's work. However, it seems to be still open whether
there is a better estimate than $O( C^2 \log C)$.
The same question arises of course for completely Sidon sets.
In particular, although unlikely to be true, it  seems  that an
   estimate $C(\Lambda)=O(C)$ is not ruled out.
    \end{rem}
    
    We will   use
  the following variant of Haagerup's well known theorem from \cite{Hainv}.
  This plays the role of the Riesz products used  in Drury's original argument (see Remark \ref{rie}).
 \begin{thm}\label{h} For any $0\le \vp\le 1$
 there is 
 a function $f_\vp:\F_\Lambda \to \C$ in $B( \F_\Lambda) $ with
 $\| f_\vp \|_{B( \F_\Lambda) } \le 1/\vp $
 such that 
 $$  \forall t\in \Lambda \quad f_\vp ( {g_t})=1\ 
 \text{ and }\ 
 \forall \gamma\not \in \{g_t\mid t\in \Lambda \} \quad \ | f_\vp(\gamma)|\le \vp.$$ 
  \end{thm} 
 \begin{proof}
  Haagerup's theorem produces 
 a unital c.p. map associated to the multiplier operator 
 for the function $h_\vp : t\mapsto  \vp^{|t|}$.
 the latter is in $B(\F_\Lambda)$ with norm 1.
  For any fixed $z\in \T$, let $\chi_z(t)= z^{n(t)}$ ($n(t)\in \Z$)
  where $t\mapsto z^{n(t)}$ is the group morphism
  on $\F_\Lambda$ taking all the generators to $z$
  (and hence their inverses to $z^{-1}$).
  Clearly $\chi_z$ has norm 1 in $B(\F_\Lambda)$.
  Therefore the function
  $$f_\vp (t) =(1/\vp) h_\vp(t) \int  \bar z \chi_z(t)  dm(z),$$
  where $m$ is normalized Haar measure on $\T$,
  satisfies by Jensen $\|f_\vp (t)\|_{B(\F_\Lambda)}\le 1/\vp$,
  $f_\vp (1)=f_\vp ({g^{-1}_t})=0   $, 
    $f_\vp (g _t)=1$ and 
    $|f_\vp (t)|\le \vp$ whenever $|t|>1$.
 All the announced properties are now easy to check.
   \end{proof}
 
   \begin{lem}\label{c99} The set
   $\tilde \Lambda =\{ g_t\mid t\in \Lambda\} \subset \F_\Lambda$
 satisfies the properties in Theorem \ref{dru} with $C=1$.
   \end{lem}
    \begin{proof} Let $z: \Lambda \to U(A)$.   
There is a unitary representation $\pi_z: \F_\Lambda \to U(A)$
such that $\pi_z(g_t)= z(t)$ for any $t\in \Lambda$.
Let $\psi_{\vp,z}(t)= f_\vp (t) \pi_z(t)$ (i.e. the pointwise product).
Then $\psi_{\vp,z}$ extends $z$, $\|\psi_{\vp,z}(t)\|\le \vp$ if $t\not\in \Lambda$
 and we claim that $\|\psi_{\vp,z}\|_{{D(\F_\Lambda,A)}} \le 1/\vp $.
Indeed, 
let $u_{\pi_z}: C^*(\F_\Lambda) \to A$  be the associated $*$-homomorphism.
Clearly
 $\|u_{\pi_z} \|_{dec}=1$
(see the proof of Proposition \ref{pd13}).
Let $M_{f_\vp}: C^*(\F_\Lambda) \to C^*(\F_\Lambda)$
be the multiplier by $f_\vp$. Then $\|M_{f_\vp}\|_{dec}\le 1/\vp$ (see Proposition \ref{pd16}).
Therefore $u_{\pi_z}M_{f_\vp} \in D(C^*(\F_\Lambda), C^*(\F_\Lambda))$
with $\|u_{\pi_z}M_{f_\vp}\|_{dec} \le 1/\vp$.
Since $u_{\pi_z}M_{f_\vp}$ is the linear map associated to the function
$\psi_{\vp,z}$ the claim follows.
    This
completes the proof in case $z$ takes its values in $U(A)$.
Using Remark \ref{dye}
one easily extends this to the case when $\sup\nl_{\Lambda}\|z  \|<1$.
    \end{proof}

  Let $\Gamma$ be another discrete group.
  
Let $T_1,T_2\in D( C^*(G), M_\Gamma)$.

 Let $T_1 \sharp T_2 : \C[G] \to \ell_1(\Gamma) $
 be defined by
 $$[T_1 \sharp T_2 ](\d_s)= \sum\nl_ {{\gamma}\in \Gamma}  T_1({\gamma},s)T_2({\gamma},s)\ e_{\gamma} ,$$
 where
 $(\d_s)$ is the natural basis of $\C[G]$ and
  $(e_{\gamma})$  the canonical basis of $\ell_1(\Gamma)$.
 Note that by \eqref{se1} the last sum is absolutely convergent.
Since $\ell_1(\Gamma)\subset C^*(\Gamma)$ (in the usual way)
 we may view $T_1 \sharp T_2$ as a map
with values in $C^*(\Gamma)$. 
Then we set equivalently  
  \begin{equation}\label{se1b} [T_1 \sharp T_2 ](\d_s)= \sum\nl_ {{\gamma}\in \Gamma}  T_1({\gamma},s)T_2({\gamma},s) U_{\Gamma}({\gamma}).\end{equation}

\begin{pro}\label{s1} For any $T_1,T_2\in D( C^*(G), M_\Gamma)$,
the mapping $T_1 \sharp T_2$ extends
to a decomposable map still denoted (abusively) by  $T_1 \sharp T_2$
in $D(   C^*(G)  , C^*(\Gamma)    )$ such that

  \begin{equation}\label{se1c} \|   T_1 \sharp T_2\|_{dec}\le \|   T_1  \|_{dec}\|     T_2\|_{dec}.\end{equation}
 \end{pro}
 \begin{proof}
Just observe
 $$        (T_1 \sharp T_2 )= P_\Gamma (T_1 \otimes T_2) J_G : C^*(G) \to   C^*(\Gamma),$$
 and use \eqref{e8}.
  \end{proof}
  
  \begin{rem}\label{s2} Assume that there is a morphism $q: \Gamma \to G$
  onto $G$ so that $G$ is a quotient of $\Gamma$.
  Let  $\hat{q}: C^*(\Gamma) \to C^*(G)$ be defined by
  $$\hat{q}( U_\Gamma({\gamma}) ) = U_G({q({\gamma})}) .$$
  Then $\hat{q}$ is a 
  $*$-homomorphism. A fortiori
  it is a c.p. contractive mapping and hence $\|\hat{q}\|_{dec}=1$.
  
   Let $T \in D(   C^*(G)  , C^*(\Gamma)    )$  such that
   $T(U_G (s)) =\sum\nl_{{\gamma}\in \Gamma} T({\gamma},s) U_\Gamma ({\gamma})$ with 
   $\sum\nl_{{\gamma}\in \Gamma} |T({\gamma},s)|<\infty$ for all $s\in G$. 
  Let $  v= \hat{q}T:  C^*(G) \to C^*(G) $. Note that
  $  v(U_G (s))= \sum\nl_{s'\in G} \sum\nl_{ {\gamma}\in \Gamma, q({\gamma})=s' } T({\gamma},s)  U_G(s'),$
  and hence $$T_v(s',s)=   \sum\nl_{ {\gamma}\in \Gamma, q({\gamma})=s' } T({\gamma},s)$$ 
  and
   $\|  v:  C^*(G) \to C^*(G)\|_{dec}  \le \| T:  C^*(G) \to C^*(\Gamma)\|_{dec} .$
 By Lemma \ref{s4} we have
 \begin{equation}\label{pp1} \|   v^\bullet  \|_{B(G)} \le \|   T  \|_{dec},\end{equation}
 and
 \begin{equation}\label{pp2} v^\bullet(s)= \sum\nl_{ {\gamma}\in \Gamma, q({\gamma})=s} T({\gamma},s).\end{equation}
    \end{rem}
    
    This brings us to the second step of the proof of Theorem \ref{dru}, as follows:
    
    \begin{lem}\label{ld13a} Let $\Lambda\subset G$  be 
 a subset generating $G$.
Let $\Gamma=\F_\Lambda$. Let  $q: \Gamma \to G$ 
be the  quotient
morphism  taking $g_t$ to $t$. If $\Lambda\subset G$ is completely Sidon with
constant $C$, there is   a scalar ``matrix"  $T(\gamma ,s)$
such that    
\begin{equation}
\label{eq99}
\sup\nl_{s\in G}\sum\nl_{\gamma\in \Gamma} |T(\gamma,s)|\le C^2
,\end{equation}
and such that the corresponding operator
$T:  C^*(G) \to  C^*(\Gamma)$ satisfies
$$\forall t\in \Lambda\quad T( U_G(t)) =U_ \Gamma(g_t).$$
Moreover, the map $\Theta: C^*(G) \to C^*(\Gamma)$ defined by
$$ \Theta( U_G(s) ) = \sum\nl_{\gamma\in \Gamma, q(\gamma)=s} T(\gamma,s) U_\Gamma(\gamma)$$
 is in $ D( C^*(G),C^*(\Gamma) )$ with $\|\Theta\|_{dec}
 \le C^2$.
 \end{lem}

\begin{proof} By Remark \ref{ko}, there is a map $\cl U: C^*(G) \to M_{\Gamma} $
with $\|\cl U\|_{dec}\le C$
such that
$ \cl U(U_G(t)) = \lambda_\Gamma(g_t)$ for all $t\in \Lambda$.
Now let
$T =   \cl U \sharp \cl U $. Then
\eqref{eq99} follows by \eqref{se1} and \eqref{se1b}.
By \eqref{se1c} $\|T\|_{dec}\le C^2$.
The second part then follows from Lemma \ref{s5b}.
 \end{proof}

 \begin{rem} By Remark \ref{07} using $\ovl{\cl U} \sharp \cl U $ we can in addition
 obtain $T({\gamma},s) \ge 0$ for all ${\gamma},s$.  
 \end{rem}
 
 \begin{rem}\label{ld13aa}
Let $\Psi: G \to C^*(\Gamma)$ be the function 
associated  to $\Theta$, i.e.
 $$\forall s\in G\quad \Psi(s)=\sum\nl_{q(\gamma)  = s}  T({\gamma},s) U_\Gamma ({\gamma}).$$
 Then $\Psi\in D(G, C^*(\Gamma))$ with $\| \Psi\|_{D(G, C^*(\Gamma))} \le C^2$
 and $\Psi(t)=U_\Gamma (g_t)$ for any $t\in \Lambda$.
 \end{rem}

 \begin{proof}[Proof of Theorem \ref{dru}]
We may assume w.l.o.g. that $G$ is the group generated by $\Lambda$.
We apply Lemma  \ref{ld13a} and \eqref{eq99} to transplant 
the result of Lemma \ref{c99} from $\F_\Lambda$ to $G$. Recall
 $\Gamma=\F_\Lambda$. Fix $0\le \vp<1$.
 Let $z: \Lambda \to A$ such that
 $\sup\nl_{\Lambda}\|z  \|<1$.
 Let $z' :  \tilde \Lambda \to A$ be the transplanted copy of $z$ defined by
 $z'(g_t)=z(t)$ for any $t\in  \Lambda$. Of course $\sup\nl_{\tilde\Lambda}\|z'  \|<1$.
 By Lemma \ref{c99} there is  $\psi'_{\vp,z}: \Gamma \to A$
 with $\| \psi'_{\vp,z}\|_{D(\Gamma , A) }  \le 1/\vp$ 
 extending $z'$ and such that $\| \psi'_{\vp,z} (\gamma)\| \le \vp$
 if $\gamma \not \in  {\tilde\Lambda}$.
 Let $u_{\vp,z}: C^*(\Gamma)  \to A$ be the linear map associated to  $ \psi'_{\vp,z} $
 (i.e. $u_{\vp,z}$ is $u_{ \psi'_{\vp,z} }$ is the sense of \eqref{97}).
 Let  $\Psi$ be associated to
 $\Theta: C^*(G)\to C^*(\Gamma) $ 
 as in Remark  \ref{ld13aa}   so that $\|\Psi\|_{D( G,C^*(\Gamma) )}
 =\|\Theta\|_{D( C^*(G),C^*(\Gamma) )}$.
 We then set
 \begin{equation}\label{98} \psi_{\vp,z}= u_{\vp,z}(  \Psi ),\end{equation}
 so that $u_{\vp,z}\Theta$ is the linear map
 associated to $\psi_{\vp,z} $. Thus
 $$\|  \psi_{\vp,z} \|_{D( G,A )} \le \|\Theta\|_{D( C^*(G),C^*(\Gamma) )} \|u_{\vp,z}\|_{D( C^*(\Gamma) , A)}=
  \|\Psi\|_{D( G,C^*(\Gamma) )} \|\psi'_{\vp,z}\|_{D(  \Gamma , A)}
  \le C^2/\vp.$$
 Equivalently \eqref{98} means that  for any $s\in G$ we have
 $$ \psi_{\vp,z}(s)=  \sum\nl_{q(\gamma)  = s}  T({\gamma},s) \psi'_{\vp,z} ({\gamma}).$$
  Observe that if
$s\not \in \Lambda$ and $q(\gamma)=s$ then necessarily 
$\gamma\not\in \{g_t\mid t\in \Lambda\}$ and hence \eqref{eq99}
gives us
$\|\psi_{\vp,z}(s) \|\le C^2\vp$.
Moreover for any $t\in \Lambda$ we have
$ \psi_{\vp,z}(t)=  u_{\vp,z}(  \Psi (t))=\psi'_{\vp,z}( g_t)=z'(g_t)=z(t)$.
So the second (and more general) part of Theorem \ref{dru} follows.
  \end{proof}
  \begin{rem}\label{rd17} Let $|s|_{\Lambda}$ denote the length of an element
  $s\in G$ with respect to the generating set $\Lambda$, i.e.
  $|s|_{\Lambda}=\inf\{|t|\mid t\in \F_\Lambda,\   q(t)=s\}$. 
  In the preceding proof
  we find $$
  |\psi_{\vp}(s)| \le C^2 \vp^{|s|_{\Lambda}-1}\text{  and  }
  \|\psi_{\vp,z}(s)\| \le C^2 \vp^{|s|_{\Lambda}-1}.$$.
   \end{rem}
    \begin{rem}\label{rie}  If one replaces the free group by the free Abelian group
   $\Gamma^a=\Z^{(\Lambda)}$ the proof becomes quite similar to Drury's original one,
   but reformulated in operator theoretic terms.
   The group $\Gamma^a$ is generated by generators $(g^a_t)_{t\in \Lambda}$
   that are free except  that they mutually commute.
   In this case $M_{\Gamma^a}$ is an injective von Neumann algebra.
   Thus we have a mapping $v \in D( C^*(G) ,M_{\Gamma^a})$ as in Corollary
   \ref{cc1} where now the $g_t$'s  are
   replaced by  the generators $g^a_t$ of $\Gamma^a$.
   When the group $G$ is Abelian
   we again have a quotient map $q^a: \Gamma^a \to G$ such that
   $q^a( g^a_t  )=t$ for all ${t\in \Lambda}$.
   The analogue of $f_\vp$ is then the Fourier transform of a probability
   measure on the compact group $\hat G=\T^\Lambda$, namely the  Riesz product 
   $\prod_{t\in \Lambda} (1+ \vp (z_t + \bar z_t)) $
   where $z_t: \T^\Lambda \to \T$ is the $t$-th coordinate. This is defined
   only for $|\vp|\le 1/2$ but one can use equally well whenever $|\vp|\le 1$
   the Riesz product based on the Poisson kernel:
   $$\prod\nl_{t\in \Lambda} (\sum\nl_{n\in \Z }   \vp^{|n|} z^n_t  ) .$$
   Its Fourier transform is the exact analogue of $f_\vp$ on $\Gamma^a$.\\
   See \cite[chap. 7]{McG} and \cite[chap. V]{HMP} for more  on Riesz products and their
   generalizations. 
   \\ See 
   \cite{Boz1,Boz2, DeM,FiPi} for  generalizations of Haagerup's result
  (concerning the function $h_\vp$)  to free products of groups 
   and \cite{Boc} for free products of c.p. maps on $C^*$-algebras.
   \end{rem}

    \section{Fatou-Zygmund property}
    We now turn to the  Fatou-Zygmund  (FZ in short) property.
    Recall  $P(G)$ is the set of positive definite complex valued functions on $G$.
    The multiplier operator $M_f$
    associated to a  function $f\in B(G)$ is c.p. on $C^*(G)$
    iff $f\in P(G)$
    and  we have $\|M_f\|=\|M_f\|_{dec}= f(1)$ for any $f\in P(G)$.
    
    \begin{thm}\label{tfz} Let $\Lambda\subset G\setminus\{1\}$ be a symmetric completely
 Sidon set.  
   Any bounded Hermitian function $\varphi: \Lambda \to \C$
   admits an extension $\tilde \varphi \in P(G)$.
   More generally, there is a constant $C'$
   such that for any unital $C^*$-algebra $A$,
   any bounded Hermitian function $\varphi: \Lambda \to A$
   admits an extension $\tilde \varphi \in CP(G,A)$
   satisfying
   $$\| \tilde \varphi (1)\|  \le C' \sup\nl_{t\in \Lambda} \| \varphi (t) \|,$$ 
   and moreover   $\tilde \varphi (1)= 1_A \| \tilde \varphi (1)\|   $.    \end{thm}
    
    The structure of the proof follows Drury's idea in \cite{dru2},
    but we again use decomposable maps as above, 
    and harmonic analysis on the free group
    instead of the free Abelian one.
    
    The key Lemma is parallel to the one in \cite{dru2}.
    It is convenient to formulate it directly
    for positive definite functions with values in
    a unital $C^*$-algebra $A$.
                      
    \begin{lem}\label{key}[Key Lemma] 
    Let $\Lambda\subset G\setminus\{1\}$ be a symmetric completely
 Sidon set with constant $C$.
Let $A$ be a unital $C^*$-algebra.
    Let $\varphi: \Lambda \to A$  be a Hermitian function
        (i.e.
        we assume  $\varphi(t^{-1})= \varphi(t)^*$ for any $t\in \Lambda$)
        with $\sup\nl_\Lambda\| \varphi \|<1$.
    For any $0\le \vp\le 1$  
    there is $\Phi_\vp \in P(G,A)$ with
    $$\|\Phi_\vp \|_{CB(G,A)} =\|\Phi_\vp (1)\| \le 4C^2\text{   and   } 
    \sup\nl_{s\in\Lambda}\|\Phi_\vp(s) -\vp \varphi(s)\|\le 4C^2 \vp^2.$$
    \end{lem}
     \begin{proof} For simplicity we give the proof assuming that
     $\Lambda$ does not contain elements such that $t=t^{-1}$.
     Let $\Lambda_1\subset \Lambda$
     be such that $\Lambda$
     is the disjoint union of
     $\Lambda_1$ and $\Lambda_1^{-1}=\{t^{-1} \mid t\in \Lambda_1\}$.
     We will work with the free group
     $\Gamma=\F_{\Lambda_1}$ instead of $\F_{\Lambda}$.
    As before we set $q(g_t)=t$ for all $t\in \Lambda_1$.
    
     Then we consider the self-adjoint
     operator space $E$
     spanned by   $\{U_G(t) \mid t\in  \Lambda\}$.
     Let $u: E \to M_\Gamma$
     be the   linear mapping defined by 
     $u(U_G(t)) = \lambda_\Gamma (g_t)$ 
     and $u(U_G(t^{-1})) = \lambda_\Gamma (g_t)^{-1}$
     for $t\in \Lambda_1$.
     Note that  $u$
    is self-adjoint in the sense that
    $u=u_*$ where $u_*(x)= u(x^*)^*$ for all $x\in E$.
By Remark \ref{ko}, since $\Lambda$ is completely Sidon with constant $C$,
    $u$ is the restriction to $E$ of a mapping
    $T
    \in D(  C^*(G) ,M_\Gamma  )$ with $\|T\|_{dec}\le C$.
    Replacing $T$ by $1/2(T+T_*)$ we may  assume
    that $T$ is self-adjoint.
    Then (see \cite{Haa}) 
    we have a decomposition $T=T^+ -T^-$
    where $T^{\pm} \in CP(   C^*(G) ,M_\Gamma  )$
    with  \begin{equation}\label{se10}   \max\{\|T^+  \|, \|T^-\|\} 
    \le
    \|T^+  + T^-\|  \le  \|T\|_{dec}.\end{equation}
  We  have
    $$T \sharp T= a-b$$
    with
    $$a= T^+ \sharp T^+ +T^- \sharp T^- \text{  and  }
    b= T^+ \sharp T^- +T^- \sharp T^+.$$ 
    Note that $a,b \in CP ( C^*(G) ,C^*(\Gamma)  )$.
    
    Fix  $0\le \vp\le 1$.
    Let $h_\vp: C^*(\Gamma) \to C^*(\Gamma)$ be as before the 
    Haagerup c.p. multiplier defined on $\F_\Lambda$  by (see \cite{Hainv})
    $h_\vp(t) = \vp^{|t|}.$
    Note that
      both $h_\vp$ and $h_{-\vp}$ are in $P(\F_\Lambda)$
      (indeed, $h_{-\vp}(t)=h_{\vp}(t) \chi_{-1}(t)$).
    
    The function $\varphi'$ defined on the words of length 1 by
    $\varphi'(g^{\pm}_t)=\varphi(t^{\pm})$ is Hermitian.
    By Haagerup's \cite{Hainv} and the
      operator valued version in \cite{Boz} (see Remark \ref{boz}), there is
    a positive definite function $f\in P(\Gamma,A)$ 
    extending $\varphi'$ such that $f(1)=1$ and 
    $f(g_t)= \varphi(t)$
    (and  $f(g^{-1}_t)= \varphi(t^{-1})$) for all $t\in \Lambda_1$.    
    Indeed, this is precisely the FZ-property
  of the free group  $\Gamma=\F_{\Lambda_1}$.
(See  
      \cite{Boc} for a  generalization of this to c.p. maps on free products.)
    Let $M_f: C^*(\Gamma) \to C^*(\Gamma) \otimes_{\max} A$
    be the associated ``multiplier"
    taking $U_\Gamma(t)$ to $U_\Gamma(t)\otimes f(t)$.
    Clearly $M_f\in CP( C^*(\Gamma), C^*(\Gamma) \otimes_{\max} A)$
    and $\|M_f\|=\|M_f(1)\|=1$.

        We now introduce for any $0\le \vp\le 1$
    $$  Y_\vp=( \hat{q}   M_{h_\vp} \otimes Id_A) M_f a   +  
    ( \hat{q}   M_{h_{-\vp}} \otimes Id_A)M_f b .$$
    Clearly $Y_\vp \in CP( C^*(G) , C^*(G) \otimes_{\max} A )$.
    Let $\Phi_\vp=  Y_\vp^\bullet$ in the sense of Remark \ref{s5}.
    Since $Y_\vp$ is c.p. we know   that
    ${\Phi_\vp}\in CP(G,A)$. Moreover, by \eqref{se10}
   $$\|\Phi_\vp \|_{CB(G,A)} =\|\Phi_\vp(1)\| \le \|a\|+\|b\|\le \| T^+ \|^2+\| T^- \|^2
   +2\| T^+ \|\| T^- \|\le 4\|T\|_{dec}^2\le 4C^2.$$
   We now compute
   $\Phi_\vp(s)$ for $s\in \Lambda$.
   We have
   $$\Phi_\vp(s) =  \sum\nl_{ \gamma\in \Gamma, q(\gamma)=s} {h_\vp}(\gamma)f(\gamma)
  (T^+(\gamma,s)^2+T^-(\gamma,s)^2)  +  {h_{-\vp}}(\gamma)f(\gamma)( 2T^+(\gamma,s)T^-(\gamma,s)). $$
   We can write (recall $s\not=1$ and hence $q(\gamma)=s$
   implies $|\gamma|\ge 1$)
   $$\Phi_\vp(s) = I(s)+E(s)$$
   where  
   $$I(s)= \sum\nl_{ \gamma\in \Gamma, q(\gamma)=s,|\gamma|=1 } {h_\vp}(\gamma)f(\gamma)
  (T^+(\gamma,s)^2+T^-(\gamma,s)^2)  +  {h_{-\vp}}(\gamma)f(\gamma)( 2T^+(\gamma,s)T^-(\gamma,s)), $$
   and the ``error term" $E(s)$ is
   $$E(s)= \sum\nl_{ \gamma\in \Gamma, q(\gamma)=s,|\gamma|>1 } {h_\vp}(\gamma)f(\gamma)
  (T^+(\gamma,s)^2+T^-(\gamma,s)^2)  +  {h_{-\vp}}(\gamma)f(\gamma)( 2T^+(\gamma,s)T^-(\gamma,s)). $$
  Fix $s\in \Lambda_1$.
  If $|\gamma|=1$ and $q(\gamma)=s$ we must have 
  $\gamma=g_s$,
  ${h_\vp}(\gamma)=\vp$ and ${h_{-\vp}}(\gamma)=-\vp$ 
  and
  $f(\gamma)=\varphi(q(\gamma))=\varphi(s)$
  so we recover
  $$I(s)= \vp \varphi(s)[ (T^+(g_s,s)^2+T^-(g_s,s)^2)-( 2T^+(g_s,s)T^-(g_s,s))]=
  \vp \varphi(s) T(g_s,s)^2,$$
  and since $T(\gamma,s)=1_{\gamma=g_s}$ we obtain for $s\in \Lambda_1$
   $$I(s)= \vp \varphi(s).$$
  Similarly,  
  $I(s^{-1})= \vp \varphi(s^{-1})=\vp  {\varphi(s )}^*$.
  \\
 It remains to estimate the error:
 Note that if $|\gamma|>1$ we have $|h_{\pm\vp}(\gamma)| \le \vp^2$ and hence
 by \eqref{se1}
 $$\|E(s)\|  \le  \vp^2 \sum\nl_{ \gamma\in \Gamma}
 |T^+(\gamma,s)^2+T^-(\gamma,s)^2|+ 2 |T^+(\gamma,s)T^-(\gamma,s) |\le \vp^2 \sum\nl_{ \gamma\in \Gamma} (|T^+(\gamma,s)|+|T^-(\gamma,s) | )^2 $$
 $$
 \le  \vp^2 (\|T^+\| + \|T^-\|)^2  \le
  4\vp^2 \|T\|_{dec}^2\le  4\vp^2 C^2. $$
  This completes the proof of the lemma, assuming $\Lambda$ has no element of order 2. Otherwise
  let $\Lambda_2\subset \Lambda$ be the set of such elements.
  We then replace $\F_{\Lambda_1}$ 
  with $\Gamma = \F_{\Lambda_1} \ast(\ast_{t\in \Lambda_2} \Z_2)$.
  We leave the details to the reader.
      \end{proof}
       \begin{rem}\label{22} Let $\varphi_0: G \to \C$ be such that
       $\varphi_0(t)=1$ if $t=1$ (unit of $G$) and $\varphi_0(t)=0$ otherwise.
       Clearly $\varphi_0 \in P(G)$ (indeed $\varphi_0(t)=\langle  \d_1 , \lambda_G(t) \d_1\rangle$).  Let $\varphi\in CP(G,A)$.
       Then $\varphi(1)\in A_+$ and hence $0\le \varphi(1) \le \|\varphi(1)\| 1_A$.
       Let $\psi(t)=\varphi(t) +\varphi_0(t)(  \|\varphi(1)\| 1_A - \varphi(1) ) $.
       Then $\psi\in CP(G,A)$, $\psi(1)= \|\varphi(1)\| 1_A $ and
       $\psi(t)=\varphi(t)$ for all $t\not= 1$. Equivalently, if we are
       given $V\in CP(C^*(G),A)$ then there is $W\in CP(C^*(G),A)$ 
       such that $W(1)= \|V(1)\| 1_A $ and
       $W(U_G(t))=V(U_G(t))$ for all $t\not= 1$.
         \end{rem}
      \begin{proof}[Proof of Theorem \ref{tfz}]
      The theorem follows from the key Lemma \ref{key}
      by a routine iteration argument (note that $\Phi_\vp -\vp \varphi$
      is Hermitian),
    exactly as in \cite{dru2}. For the last assertion we use Remark \ref{22}.
        \end{proof}

        The proof gives an estimate
        of the form $C'\le c C^4$ where $C$ is the completely Sidon constant
        and $c$ a numerical constant, to be compared with Remark \ref{mela}.

            \begin{cor}\label{da1}  Assume for simplicity that    $\Lambda\subset G\setminus\{1\}$ is symmetric, 
        and is the disjoint union of
        $\Lambda_1$ and $\Lambda^{-1}_1$ as before (in particular it has no element of order 2).
        Let $E_{\Lambda}\subset C^*(G)$
        be the operator system generated by $\Lambda$ and $\{1\}$.
            The following are equivalent: 
            \item(i) $\Lambda$ is completely Sidon.
            \item(ii) There is a completely positive linear map
        $V:  C^*(G) \to C^*(\F_{\Lambda_1})$ such that
$$\forall t\in {\Lambda_1} \quad V(U_G(t))=U_{\F_{\Lambda_1}}(  g_t) \quad  V(U_G(t^{-1})) =   U_{\F_{\Lambda_1}}(  g^{-1}_t).$$  
            \item(iii) There is $\d>0 $
            such that the ({unital}) mapping
            $S_\d : E_\Lambda \to C^*({\F_{\Lambda_1}})$
            defined by 
            $$
            S_\d (1)=1 \quad{and} \quad\forall t\in {\Lambda_1} \quad S_\d (U_G(t))=\d U_{\F_{\Lambda_1}}(  g_t) \quad  S_\d (U_G(t^{-1})) =   \d U_{\F_{\Lambda_1}}(  g^{-1}_t),$$
            is c.p.. 
            \item(iv)  There is $\beta>0 $
            such that $S_\beta$ admits
            a c.p. extension  
        $\tilde{S_\beta}:  C^*(G) \to C^*(\F_{\Lambda_1})$.    \\
            Moreover, the relationships between the Sidon constant
            and $\d$ are $C\le 1/\d \le c C^4$, and $\beta\ge \d^2$.  
               \end{cor}
          \begin{proof}  
          Assume (i). Let $A=C^*(\F_{\Lambda_1})$.
          Define $\varphi: \Lambda \to A$ by
          $\varphi(t)= g_t, \varphi(t^{-1})= g^{-1}_t$ for $t\in \Lambda_1$.
          By Theorem \ref{tfz}  
          there is a c.p. mapping $V: C^*(G) \to A$
        extending $U_G(t) \mapsto \varphi(t)$. This proves (i) $\Rightarrow$ (ii).
         Assume (ii). 
         Let $\d = \|V(1)\|^{-1}$.  By Remark \ref{22} there is
        $W \in CP( C^*(G) , C^*(\Gamma))$ 
          such that $W(1)= (1/\d) 1$ and
         $\forall t\in {\Lambda_1} \quad  W(U_G(t))= U_{\F_{\Lambda_1}}(  g_t) \quad  W (U_G(t^{-1})) =     U_{\F_{\Lambda_1}}(  g^{-1}_t).$
         Then     the restriction $S_\d$ of $\d W$   to $E_{\Lambda}$ satisfies (iii).\\
         Assume (iii) or (iv). Then (i) follows
         because $\|S_\d\|_{cb}= 1$. 
         Also (iv) trivially implies (iii).\\
         Assume (iii). Let $\Gamma =   {\F_{\Lambda_1}}$.
         By Remark \ref{ko}
         $S_\d$ extends to a c.p. map   $\cl U: C^*(G)\to M_\Gamma$. 
         Now consider $ S= \cl U \sharp \cl U$.
         Then $S$ is c.p. and extends $S_{\d^2}$. Thus (iii)  implies (iv).\\         
        The relationships between the constants
         can be traced back easily from the proof. 
          \end{proof}
\begin{rem}  All the preceding can be developed
in parallel for the free Abelian group.
The last statement gives an apparently new fact
(or rather, say, a new reformulation of the FZ property)
in the commutative case.
We state it for emphasis because it seems interesting.
Let $G$ be a discrete commutative group.
 Assume for simplicity that    $\Lambda\subset G\setminus\{0\}$ has no element of order 2
         and is the (symmetric) disjoint union of
        $\Lambda_1$ and $\Lambda^{-1}_1$ as before.
        Let $\Gamma_1$ be the free Abelian group  $\Z^{(\Lambda_1)}$.
        Note $C^*(\Gamma_1)\simeq C(\T^{\Lambda_1})$.
        Then $\Lambda$ is Sidon iff
        there is $\d>0$ such that the mapping
        $$S_\d: E_\Lambda \to C^*(\Gamma_1)\simeq C(\T^{\Lambda_1})$$
        defined as above but with  $\Z^{(\Lambda_1)}$
        in place of $\F_{\Lambda_1}$ is positive.
        Note that in the commutative case positive
        implies c.p..
 \end{rem}
 
 \section{Characterizations by operator space properties}
 
 Let $\Lambda\subset G$ be a subset and let
 $C_\Lambda\subset C^*(G)$ be its closed linear span.
In the classical  setting,  when $G$ is a commutative discrete group,
Varopoulos \cite{V} proved that $\Lambda $ is Sidon as soon as
$C_\Lambda$ is isomorphic to $\ell_1(\Lambda)$
as a Banach space (via an arbitrary isomorphism).
Shortly after that, the author and independently Kwapie\'n and Pe\l czy\'nski
proved that it suffices to assume     that $C_\Lambda$ is of cotype 2.
This was refined by Bourgain and Milman \cite{BMs}
who showed that
$\Lambda$ is Sidon if (and only if) 
$C_\Lambda$ is of finite cotype.
It is natural to try to prove analogues of these results
for a general discrete group $G$.
The next statement shows that if $C_\Lambda$ is
\emph{completely}  isomorphic to $\ell_1(\Lambda)$
(equipped with its maximal operator space structure)
then $\Lambda$ is completely Sidon.
Indeed, the dual operator space $C_\Lambda^*$
is then {completely}  isomorphic to $\ell_\infty(\Lambda)$,
and the latter is exact with constant 1.

We recall that an operator space (o.s. in short) $X\subset B(H)$
is called exact if there is a constant $C$ such that for any finite dimensional
subspace $E\subset X$ there is an integer $N$, a subspace $\tilde E\subset M_N$ and 
an isomorphism $u: E \to \tilde E$ such that $\|u\|_{cb}  \|u^{-1}\|_{cb} \le C$.
The smallest constant $C$ for which this holds
is denoted by ${\rm ex}(X)$.

The dual o.s. of an o.s. $X\subset B(H)$ 
is characterized by the existence of an isometric embedding $X^*\subset B(H)$
such that the
natural norms on the spaces $M_n(X^*)$ and $CB(X,M_n)$
coincide. See \cite[\S 2.3]{P4} for more on this.

\begin{thm}\label{tt1} If $C_\Lambda^*$ is an exact operator space,
then  $\Lambda$ is completely Sidon with constant
$4{\rm ex}(C_\Lambda^*)^2$. Conversely,
if $\Lambda$ is completely Sidon with constant $C$ then
then ${\rm ex}(C_\Lambda^*)\le C$.
\end{thm}
\begin{proof} The converse part is clear because
$\ell_\infty(\Lambda)=\ell_1(\Lambda)^*$ is exact with
 ${\rm ex}(\ell_\infty(\Lambda))=1$.\\
 Assume that $C_\Lambda^*$ is exact.
Let $\alpha\subset \Lambda$ be a finite subset.
Consider the mapping
$T_0: C_\Lambda \to C^*_\lambda(\F_\Lambda)$
defined by 
$T_0(t)= \lambda_{\F_\Lambda}(g_t)$ for $t\in \alpha$
and $T_0(t)= 0$ for $t\not\in \alpha$.
Let us denote by $\varphi_t\in ( C^*(G))^*$ the functional
biorthogonal to the natural system, i.e.
$\varphi_t(U_G(s))=\d_t(s)$.

Let $a: G \to M_N$ be a finitely supported $M_N$-valued
function ($N\ge 1$).
We have then by elementary arguments
\begin{equation}\label{cc6}\| \sum a(t) \otimes U_G(t) \| \ge  
\| \sum a(t) \otimes \lambda_G(t) \|
\ge \max\{\|\sum a(t)^*  a(t)\|^{1/2}, \|\sum a(t)  a(t)^*\|^{1/2}\}.
\end{equation}
By a well known inequality with roots in Haagerup's   \cite{Hainv}
(see  \cite[p. 188]{P4})
\eqref{cc6} implies
\begin{equation}\label{cc7}\| \sum a(t) \otimes U_G(t) \| \ge  
(1/2)\| \sum a(t) \otimes \lambda_{\F_\Lambda} (g_t) \|
\end{equation}
  and hence $\|T_0\|_{cb}\le 2$.
Equivalently this means that the tensor
$${ \bf T}_0=\sum\nl_{t\in \alpha} \varphi_t \otimes \lambda_{\F_\Lambda}(g_t)\in 
(C_\Lambda)^* \otimes C^*_\lambda(\F_\Lambda)$$
satisfies
\begin{equation}\label{cc2}\|{\bf  T}_0\|_{\min}=\|T_0: C_\Lambda \to C^*_\lambda(\F_\Lambda) \|_{cb} \le 2.\end{equation}
Let $\vp>0$.  Assume $|\alpha|=n$ and $\alpha=\{t(1),\cdots,t(n)\}$. Let $\Gamma\subset {\F_\Lambda}$
be the copy of $\F_n$ generated by $\{g_{t(j)}\mid 1\le j\le n\}$.
We claim that
$T_0$  extends to an operator $\tilde T: C^*(G) \to M_\Gamma$
such that $\|\tilde  T\|_{dec}\le 2 {\rm ex}(X) (1+\vp)$.

By  a result
due to Thorbj{\o}rnsen and Haagerup \cite{HT3}
(see \cite[p. 331]{P4}) recently refined in \cite{CM}
we have (here we denote by $(g_j)$ the free generators of $\F_n$):\\
 {\it For any $n$ and $N$ there is an $n$-tuple of $N\times N$-unitary matrices
$(u^{(N)}_j)_{1\le j\le n}$ such that
for any exact operator space $X$ and any $x_j\in X$ we have
\begin{equation}\label{cc3}\lim\nl_{N\to \infty} \| \sum u^{(N)}_j \otimes x_j \|_{M_N(X)}\le {\rm ex}(X)  \|\sum \lambda_{\F_n}(g_j) \otimes x_j \|_{\min},\end{equation}
and
\begin{equation}\label{cc5} \{ u^{(N)}_j  \mid 1\le j\le n\} \text{ converges in moments to }
\{  \lambda_{\F_n}(g_j) \mid 1\le j\le n\} .\end{equation}
}\\
Let $X=C_\Lambda^*$. This gives us by \eqref{cc2}
$$\lim\nl_{N\to \infty} \| \sum u^{(N)}_j \otimes \varphi_{t(j)} \|_{M_N(X)} 
\le 2 {\rm ex}(X).$$
For some $n_0$ we have
$$\sup\nl_{N\ge n_0} \| \sum u^{(N)}_j \otimes \varphi_{t(j)} \|_{M_N(X)} 
\le 2 {\rm ex}(X) (1+\vp).$$
This gives us a map $T_1: C_\Lambda \to (\oplus\sum\nl_{N\ge n_0}  M_N)_\infty$
with $\|T_1\|_{cb} \le 2 {\rm ex}(X) (1+\vp)$, 
such that  $T_1( U_G({t(j)}))= \oplus_{N\ge n_0}  u^{(N)}_j$.
Let $\omega$ be a nontrivial ultrafilter on $\N$.
By \eqref{cc5}, we have an isometric embedding
$M_{\Gamma}\subset  (\oplus\sum\nl_{N\ge 1}  M_N)_\infty/\omega$
and a surjective unital c.p. map $Q_\omega: (\oplus\sum\nl_{N\ge 1}  M_N)_\infty
\to M_\Gamma$,
such that
$$Q(  \oplus_{N\ge n_0}  u^{(N)}_j )=
 \lambda_{\Gamma}(g_j).$$ Since $(\oplus\sum\nl_{N\ge 1}  M_N)_\infty$ is injective
 there is an extension of $T_1$
 denoted $\tilde T_1: C^*(G) \to (\oplus\sum\nl_{N\ge 1}  M_N)_\infty$ such that
$\|\tilde T_1\|_{dec}=\|\tilde  T_1\|_{cb}\le \|T_1\|_{cb} \le 2 {\rm ex}(X) (1+\vp)$,
and hence setting
$\tilde T= Q \tilde T_1$,
we obtain the claim. 
Then we conclude by Corollary \ref{cc1}.
\end{proof}
\begin{cor}   Let $\Lambda\subset G$.
 The operator space $C_\Lambda\subset C^*(G)$ is completely isomorphic
 to $\ell_1(\Lambda)$ (with its maximal o.s. structure) iff
 $\Lambda$ is completely Sidon.
\end{cor}
\begin{rem} By the same argument, we can replace the exactness assumption of Theorem
\ref{tt1}
by the subexponentiality 
(or tameness) in the sense of \cite{P137}.
\end{rem}

\begin{rem} By the same argument,  
the following can be proved.
Let $\{x_j\} \subset A$ be a bounded sequence in a $C^*$-algebra $A$.
Assume that for some constant $c$, for any $N$ and  any sequence $(a_j)$ 
in $M_N$ with only
finitely many nonzero terms we have
$$c\|\sum a_j \otimes x_j\|\ge \max\{\|\sum a_j^*  a_j\|^{1/2}, \|\sum a_j  a_j^*\|^{1/2}\}.$$
Let $E$ be the closed span of $\{x_j\} $.
If  $E^*$ is exact
then $\{x_j\otimes x_j\} $ is  completely Sidon
in $A\otimes_{\max} A$ (with constant $4c^2 {\rm ex}(E^*)^2$). See \cite{Pi8}
for more on that theme.
\end{rem}

\begin{rem}  
\item{(i)} Let us first observe that the Varopoulos result mentioned above
remains valid for a    non-commutative group $G$.
We will show that if $C_\Lambda$ is isomorphic
to $\ell_1(\Lambda)$, 
then the usual linear mapping taking the canonical basis of $\ell_1(\Lambda)$, namely $(\d_t)_{t\in \Lambda}$, 
  to $(U_G(t))_{t\in \Lambda}$ is an isomorphism.
Actually it suffices to assume that  $C_\Lambda^*\simeq \ell_\infty(\Lambda)$ as a Banach space
or that, say, $C_\Lambda^*$ is a ${\cl L}_\infty$-space, or that $(C_\Lambda^*,C_\Lambda^*)$ is a GT-pair
in the sense of \cite[Def. 6.1]{P133}, to which we refer for all unexplained
terminology in the sequel.

With the preceding notation,
let 
$W_x: C^*_\Lambda \to C_\Lambda$ be the  linear operator associated
to the tensor   $x=\sum\nl_{t\in \alpha}   x(t) U_G(t)  \otimes  U_G(t)\in C_\Lambda\otimes C_\Lambda$.
Let $\|\ \|_\vee$ be the norm in the injective tensor product (in the
usual Banach space sense)
of ${C^*(G)}$ with itself.
Note 
$$\|W_x\|=\|x\|_{\vee} \le  \|x\|_{\min}= \| \sum\nl_{t\in \alpha}   x(t) U_G(t) \|_{C^*(G)}.$$
Let $(z(t))\in \T^\Lambda$. 
Let $T_z: C_\Lambda \to C^*_\Lambda$ be the  linear operator associated
to the tensor   $\sum\nl_{t\in \alpha}   z(t) \varphi_t \otimes  \varphi_t \in C_\Lambda^*\otimes C_\Lambda^*$.
A simple verification shows that,
denoting by $\gamma_2(T_z)$  the norm of factorization through Hilbert space
of $T_z$,  we have  $\gamma_2(T_z)\le 1$.

Then  Grothendieck's Theorem, or  
our Banach space assumption (see \cite[\S 6]{P133}), implies
that
   for any finite rank map
$w:   C^*_\Lambda \to C_\Lambda$ we have 
$|\tr(wT_z)|\le K \gamma_2(T_z) \|w\|_\vee\le K   \|w\|$, where $K$ is a  constant
independent of $w,z$.
Therefore,
we have
$$| \sum\nl_{t\in \alpha}   x(t) z(t)|=|\tr(W_x T_z)|\le K   \|x\|_{C^*(G)},$$
and hence taking the sup over all $z$'s and $\alpha$'s
$$  \sum\nl_{t\in \Lambda}   |x(t)| \le K   \|x\|_{C^*(G)}.$$
Thus  we conclude that  $C_\Lambda$ is isomorphic to $\ell_1(\Lambda)$
by the usual (basis to basis) isomorphism.
Such sets are called weak Sidon
   in \cite{Pic}, where the term Sidon is reserved for the sets that span
   $\ell_1(\Lambda)$ in the reduced $C^*$-algebra $C^*_\lambda(G)$.

\item{(ii)} Let $C^\lambda_\Lambda  $ be the closed span
of $\Lambda$ in $C^*_\lambda(G)$, i.e.
$C^\lambda_\Lambda=\ovl{\rm span}  \{\lambda_G(t)\mid t\in \Lambda\} $.
The preceding argument applies equally well to $C^\lambda_\Lambda  $,
and shows that if $C^\lambda_\Lambda  $ is isomorphic to
$\ell_1(\Lambda)$ (by an arbitrary isomorphism) then
it actually is so by the usual isomorphism, and  
 $\Lambda$ is Sidon in the sense of \cite{Pic}.

\item{(iii)} Lastly, we   apply
the same idea to slightly generalize Theorem \ref{tt1}.\\  Fix $N\ge 1$. Let $z=(z(t))\in U(N)^\Lambda$
and $x=(x(t))\in M_N^\Lambda$.
Consider  the tensors  $${\bf T_z}= \sum\nl_{t\in \alpha}    \varphi_t \otimes [z(t) \otimes \varphi_t] \in C_\Lambda^*\otimes M_N(C_\Lambda^*),$$
and
 $${\bf W_x}=\sum\nl_{t\in \alpha}   x(t) \otimes U_G(t)  \otimes  U_G(t)\in 
M_N(C_\Lambda \otimes  C_\Lambda).$$
Then it can be checked  on the one hand that
$$\max\{ \|  {\bf T_z}\|_{ C^*_\Lambda \otimes_h M_N(C^*_\Lambda)}  
,
 \|  {\bf {}^t T_z}\|_{ M_N(C^*_\Lambda)\otimes_h C^*_\Lambda}
   \} \le 1.$$
Thus if the pair $ (C_\Lambda^*,M_N(C_\Lambda ^*) ) $
satisfies (uniformly over $N$)  the o.s. version of Grothendieck's theorem described in
\cite[Prop. 18.2]{P133} we find for some constant $K$ (independent of $N$)
$$ \|  {\bf T_z}\|_{ C^*_\Lambda \otimes_\wedge M_N(C^*_\Lambda)}  
  \le K.$$
  Here  $\otimes_\wedge$ is the  projective tensor product in the operator space
sense.
  A fortiori, this implies
  $$ \|  {\bf T_z}\|_{M_N( C^*_\Lambda \otimes_\wedge C^*_\Lambda)}  
  \le K.$$
  On the other hand, we have obviously
  $$\|{\bf W_x} \|_{  M_N(C_\Lambda \otimes_{\min}  C_\Lambda)  }   
  \le \| \sum\nl_{t\in \alpha}   x(t) \otimes U_G(t) \|_{  M_N(C_\Lambda)}.$$
  Thus we obtain
  $$\| \sum z(t) \otimes x(t)  \|\le 
   \|  {\bf T_z}\|_{M_N( C^*_\Lambda \otimes_\wedge C^*_\Lambda)} 
   \|{\bf W_x} \|_{  M_N(C_\Lambda \otimes_{\min}  C_\Lambda)  } 
   \le K \| \sum\nl_{t\in \alpha}   x(t) \otimes U_G(t) \|_{  M_N(C_\Lambda)}.$$
   The latter implies that $\Lambda$ is completely Sidon.
   \end{rem}

 \section{Remarks and open questions}\label{quest}

   \subsection{\bf Free sets}

  We start by the characterization of the case $C=1$ announced in Proposition
  \ref{pd13}.
   \begin{pro}\label{p02} The following properties of a subset $\Lambda \subset G$ are equivalent:
 \begin{itemize}
 \item[{\rm (i)}]   $\Lambda$   is
completely Sidon with a constant $C=1$.
   \item[{\rm (ii)}] For any finite subset $S\subset \Lambda$
  we have $\|\sum\nl_{s\in S}\lambda_G(s) \|=2\sqrt{|S|-1}.$  
   \item[{\rm (iii)}] $\Lambda$ is a (left say) translate of a free set enlarged by including the unit.
   \item[{\rm (iv)}] For every   $m$ and every $2m$-tuple $t_{1}, t_{2},t_{3},\cdots,t_{2m-1},t_{2m}$
 in $\Lambda$ with $ t_{1}\not = t_{2}\not = \cdots t_{2m-1} \not =t_{2m} $ we have
 $t_{1}^{-1} t_{2} t_{3}^{-1}t_4\cdots t_{2m-1}^{-1}t_{2m} \not=1.$
 \end{itemize}
 \end{pro}
  
\begin{proof} We start by
(iii) $\Rightarrow$ (i).  Assume (iii). Since  translation has no significant effect,
  it suffices to prove (i) for  $\Lambda=S \cup \{1\}$ with $S$ free.
  We may assume that $S$ generates $G$.
  Let $z: \Lambda \to U(A)$ such that $z(1)=1$. By the freeness of $S$
  there is a unitary representation $\pi: G \to A$ extending $z$.
By Remark \ref{rd13} $\Lambda$ is completely Sidon set with $C=1$.
Conversely, let us show (i) $\Rightarrow$ (iii).\\
Assume (i). Pick and fix an element $s\in \Lambda$.
We may assume after (left say) translation by $s^{-1}$ 
that $1\in \Lambda$.  Then the correspondence
$t\mapsto g_t$ ($t\not=s$) extends to a unital completely contractive map
from the span of $\Lambda$ in $C^*(G)$ to that
of $\{1\}\cup\{g_t\mid t\in \Lambda\setminus\{s\} \}$ in $C^*(\F_\Lambda)$.
By \cite[Prop. 6]{Pki} the latter mapping is the restriction of a unital $*$-homomorphism
from  $C^*(G)$ to $C^*(\F_\Lambda)$, which (by the maximality
of  $C^*(\F_\Lambda)$) must be a $*$-isomorphism.
Translating back by $s$ yields (iii). \\
 (iii) $\Leftrightarrow$ (iv) is due to 
Akemann-Ostrand \cite[Def. III.B and Th. III.D]{AO}, as well
as  (iii) $\Rightarrow$ (ii) and the converse is due to Lehner \cite{Leh}.   
\end{proof}

Since free sets (or their left or right translates) are the fundamental 
completely Sidon examples,  and the latter are stable by finite unions
it is natural to ask: Is any 
completely Sidon set a finite union
of  translates of free sets ?
In other words (see Proposition \ref{p02}): is every 
completely Sidon set with constant $C<\infty$ a finite union of
sets with   $C=1$ ?
Of course this would imply  that any group $G$ that contains
an infinite completely Sidon set  contains a copy of $\F_\infty$
as a subgroup, but we do not even know
  whether this is true, although non-amenability is known (see
  Remark \ref{r22}).
\begin{rem} In \cite{P9}  we asked whether an $L$-set (see the definition below) is a finite union of left translates of free sets, but Fendler gave a simple counterexample
in Coxeter groups  in \cite{Fen}.
  \end{rem}

\subsection{\bf $L$-sets}

In \cite{P9} (following \cite{HP2})
we study a class of subsets of discrete groups that we call $L$-sets.
By definition, $L$-sets are the sets satisfying \eqref{e111} below.
These sets are the same as those called strong 2-Leinert sets in \cite{Boz3}.
$L$-sets seem to be somehow the \emph{reduced} $C^*$-algebraic analogue of our  
completely Sidon sets.
Indeed, $\Lambda\subset G$ is an $L$-set
iff the linear map taking 
$\lambda_{\F_\Lambda}(g_t)$ to $t\in \Lambda$
extends to a complete isomorphism $v$
   from the span
of $\tilde \Lambda$ in $C^*_\lambda(\F_\Lambda)$
to that of $\Lambda$ in $C^*_\lambda(G)$. 
If \eqref{e111} holds we have
$\|v\|_{cb}\le C'$
and  $\|v^{-1}\|_{cb}\le 1$ always holds.
 The connection
between completely Sidon sets
and $L$-sets is unclear. However our  Proposition \ref{p2} below
  suggests that completely Sidon sets
are probably $L$-sets. 

\begin{pro}\label{p2} Assume that $C^*_\lambda(G)$ is an exact $C^*$-algebra
($G$ is then called an ``exact group").
Let $\Lambda\subset G$ be a completely Sidon set.
There is a constant $C'$ such that
for any $k$ and any finitely supported
function $a: \Lambda \to M_k$ we have
\begin{equation}\label{e111}\|\sum\nl_{t\in \Lambda} a(t) \otimes \lambda_G(t)\|
\le C' \max\{ \|\sum a(t)^*a(t)\|^{1/2}, \|\sum a(t)a(t)^*\|^{1/2}    \}.\end{equation}
In other words
$\Lambda$ is an $L$-set in the sense of \cite{P9} .
\end{pro}
 \begin{proof} Fix $k$. Let $(U_t)_{t\in \Lambda}$ be an i.i.d. family
 of random matrices uniformly distributed in the unitary group $U(k)$.
 Let $z(t)=\ovl{U_t}$.
  By  (ii) in Proposition \ref{pfz} we have $\|u_z\|_{cb}\le C$ and hence
 \begin{equation}\label{e113}\|\sum\nl_{t\in \Lambda}  [a(t) \otimes \lambda_G(t)\otimes  {U_t}  ] \otimes \ovl{U_t}\|\le
 C\|\sum\nl_{t\in \Lambda}   
 [a(t) \otimes \lambda_G(t)\otimes  {U_t}  ]  \otimes  U_G(t)\|.\end{equation}
  Since $U_G  \otimes \lambda_G$
is equivalent to $\lambda_G$ (by Fell's absorption principle, see e.g. \cite[p. 149]{P4})
and we may permute the factors
$$\|\sum\nl_{t\in \Lambda}   
 [a(t) \otimes \lambda_G(t)\otimes  {U_t}  ]  \otimes  U_G(t)\|
 =
 \|\sum\nl_{t\in \Lambda}   
 a(t) \otimes   \lambda_G(t) \otimes  {U_t}\|,$$
 and   since the operators $U_t \otimes \ovl{U_t}$
 have a common eigenvector
 $$\|\sum\nl_{t\in \Lambda}   a(t)\otimes   \lambda_G(t)\|\le
 \|\sum\nl_{t\in \Lambda}  [a(t) \otimes \lambda_G(t)\otimes  {U_t}  ] \otimes \ovl{U_t}\|.$$
 Therefore \eqref{e113} implies
 $$\|\sum\nl_{t\in \Lambda}   a(t)\otimes   \lambda_G(t)\|\le
 C\|\sum\nl_{t\in \Lambda}   
 a(t) \otimes   \lambda_G(t) \otimes  {U_t}\|.$$
We now recall that the matrices ${U_t}$ are random $k \times k$ unitaries and
  we let $k\to \infty$.
  By \cite{CM} (actually   \cite[Th. B]{HT2} suffices for our needs) the announced inequality follows with $C' = 2C$.
 \end{proof}
 
 \begin{rem} In Proposition \ref{p2} it clearly suffices to assume that 
$C^*_\lambda(G)$  is ``completely tight" or ``subexponential" in the sense
of \cite{Pirm}.
 \end{rem}
  \begin{rem} We refer to \cite[\S 9.7]{P4} for all the terms used here.
   By Remark \ref{har} below applied with $p=1$,
  if $\Lambda\subset G$ (assumed infinite for simplicity) is completely Sidon, then the span
  of $\Lambda$ in  $L_1(\tau_G)={M_G}_*$ is completely isomorphic
  to the operator space $R+C$. But we see no reason why
 it should be  completely complemented in  $L_1(\tau_G)$,
 so we do not see how to deduce from this that the span  of $\Lambda$ in $M_G$
 or in $C^*_\lambda(G)$ is completely isomorphic
  to the operator space $R\cap C=(R+C)^*$.
 \end{rem}
  
 Note that the question whether $C^*_\lambda(G)$ is an exact $C^*$-algebra
 for all groups $G$
 remained open for a long time, until Ozawa \cite{Oz}
 proved that a group constructed by Gromov in \cite{Gro}
  (the so-called ``Gromov monster") 
 is a counterexample. See also \cite{ArzD} and also \cite{Osa1,Osa2}
 for more recent examples.
 This shows that the assumption that $G$ is exact in
  Proposition \ref{p2} is a serious restriction,
  although it holds in many examples.
 
 In the converse direction
 we do not have any example
at hand of an $L$-set that is not completely Sidon.

\subsection{\bf $\Lambda(p)$-sets}
In \cite{Boz1b,Boz1c} Bo\.zejko considered the analogue of Rudin's
$\Lambda(p)$-sets in a non-abelian discrete group $G$. He proved 
that any  sequence in $G$ contains a subsequence
forming a $\Lambda(p)$-set with $\Lambda(p)$-constant
growing like $\sqrt p$ (we  call such sets ``subgaussian" in \cite{Pi3}).
In this direction,
a natural question arises: which sequences in $G$ contain
a completely Sidon subsequence ?
similarly, which contain a   subsequence forming an $L$-set ?
Obviously this is not true for any infinite sequence. It 
seems interesting to understand the underlying combinatorial 
(or operator theoretic) property
that allows the extraction.
In this context, we recall Rosenthal's famous dichotomy \cite{Ros}
for a sequence in a Banach space: it contains either 
a weak Cauchy subsequence or a $\ell_1$-sequence (i.e.
the analogue of a Sidon sequence). Is there an operator space analogue
of Rosenthal's theorem ?

\subsection{\bf $\Lambda(p)_{cb}$-sets}

$L$-sets are also $\Lambda(p)_{cb}$-sets
in the sense of Harcharras \cite{Har} for any $2<p<\infty$. In fact    $L$-sets
are just  $\Lambda(p)_{cb}$-sets  
with uniformly bounded $\Lambda(p)_{cb}$-constant when $p\to \infty$.
 We refer to \cite{Har}  for more information on these operator
space analogues of Rudin's $\Lambda(p)$-sets.

\begin{rem}\label{har} If $\Lambda\subset G$ is completely Sidon, then a fortiori
  it is ``weak Sidon" in the sense of \cite{Pic}.
  This means that any 
  bounded scalar valued   function
  on $\Lambda$ is the restriction
  of a  multiplier in $B(G)$. Since the latter are 
    c.b.  multipliers
  on $L_p(\tau_G)$ simultaneously for all $1\le p< \infty$
(by Proposition \ref{pd16} and complex interpolation)
we can use the Lust-Piquard-Khintchine inequalities (see \cite[p. 193]{P4})
to show that for any $1\le p<\infty$
the span of $\Lambda$ in $L_p(\tau_G)$
is isomorphic to that of $\tilde\Lambda$ in $L_p(\tau_{\F_\Lambda})$.
Therefore, $\Lambda$ is $\Lambda(p)_{cb}$ for any $2<p<\infty$ and
  the corresponding constant  is $O(\sqrt p)$ when $p\to \infty$.
 Such sets could be called
  ``completely  subgaussian". 
  Whether conversely the $\Lambda(p)_{cb}$-constant being $O(\sqrt p)$
  implies weak Sidon
     probably fails
  but we do not have any counterexample. It is natural to ask whether
   this ``completely  subgaussian" property implies
  that the set defines an unconditional basic sequence
  in the \emph{reduced} $C^*$-algebra of $G$.
  In this form this is correct for commutative groups 
  by our result from 1978 (see \cite{Pi3}), but
  what about amenable groups ?
  
 \n In  \cite{BGM}
 it is proved   that the generators
  in any Coxeter group satisfy the    weak Sidon property
  and  the preceding remark is explicitly applied to that case.
 \end{rem}

\subsection{\bf Exactness}\label{ex}

It is a long standing problem raised by Kirchberg whether the exactness of
the full $C^*$-algebra $C^*(G)$ of a   discrete group $G$
implies the amenability of $G$.
We feel that the preceding results may 
shed some light on this.

Let $\Lambda \subset A$ be a subset of a $C^*$-algebra $A$.
Let ${\F_\Lambda}$ be the free group
with generators $(g_t)$ indexed by ${\Lambda}$.
Following \cite{Pi8} we say that   $\Lambda \subset A$
is completely Sidon with constant
$C$ if the linear map
taking $t\in \Lambda$ to $U_{\F_\Lambda}(g_t)$
is c.b. with c.b-norm $\le C$.\\
For any $n\ge 1$, let $\Lambda_n$ be 
linearly independent \emph{finite} sets in the unit ball of   $A$
with  $|\Lambda_n|\to \infty$.
Let $C(\Lambda_n)$ be the completely Sidon constant.
By  \cite[Th. 21.5, p. 336]{P4}
if $C(\Lambda_n)=o(\sqrt{ |\Lambda_n|})$
then $A$ cannot be exact.
In particular, if this holds for  $A= C^*(G)$
then $G$ is not amenable. A fortiori,
if $A= C^*(G)$ contains an infinite completely Sidon set
then $G$ is not amenable.

Thus one approach to the preceding Kirchberg problem
could be to show conversely that if $G$ is non-amenable
then there is a sequence $(\Lambda_n)$ of such  sets   in $A= C^*(G)$
or even in $G$.

The analogous fact for the reduced $C^*$-algebra
was proved by Andreas Thom \cite{Th}.

\subsection{\bf Interpolation sets}

Sidon sets are examples
of ``interpolation sets". Given an abstract set $G$
given with a space $X\subset \ell_\infty(G)$ of functions on $G$,   a subset $\Lambda \subset G$
is called an interpolation set for $X$ if any bounded function on $\Lambda$
is the restriction of a function in $X$.

It is known (see \cite{P9}) that  $\Lambda \subset G$  is   an $L$-set iff
any  (real or complex) function bounded
on $\Lambda$  and vanishing outside  it  is a c.b. (i.e. ``Herz-Schur")
multiplier on the von Neumann algebra of $G$. 
In other words
$\Lambda$ is an interpolation set for the class of such multipliers,
\emph{with an additional property:} that the indicator function of $\Lambda$
is also a  a c.b. (Herz-Schur)
multiplier.

  In \cite{Pic} Picardello introduces
  the term ``weak Sidon set"  
for a subset $\Lambda \subset G$
such that any bounded function on $\Lambda$
  is the restriction of one in $B(G)=C^*(G)^*$. In other words, $\Lambda$ is an interpolation set for $B(G)$.
  By Hahn-Banach this is 
  the same as saying that the closed span of $\Lambda$ in the full 
  $C^*$-algebra $C^*(G)$ is isomorphic as a Banach space to $\ell_1(\Lambda)$
  by the natural correspondence.\\
 In \cite{Pic}  the term Sidon (resp. strong Sidon) is then
 (unfortunately in view of our present work) reserved  
for the interpolation sets for $B_\lambda(G)=C_\lambda^*(G)^*$
(resp. for the sets such that any 
function in $c_0(\Lambda)$ extends to one in $A(G)$).
Simeng Wang observed recently in \cite{Wang} that Sidon and strong Sidon in Picardello's sense
are equivalent.

\begin{rem}[Operator valued interpolation] 
A subset $\Lambda\subset G$ is completely Sidon 
iff it is an interpolation set for operator valued functions
more precisely iff any bounded $B(H)$-valued function on $\Lambda$
is the restriction of one in $CB(G, B(H))$.
 Indeed, this is Proposition \ref{pfz}.
 Moreover, if this holds then by Theorem \ref{tfz} for any  unital $C^*$-algebra $A$
 any bounded $A$-valued function on $\Lambda$
is the restriction of one in $D(G,A)$.
\end{rem}
\begin{rem}[Final remark] In \cite{Pi8}  we prove  a version of the union theorem
for subsets of a general $C^*$-algebra $A$.
We can recover the group case when $A=C^*(G)$.
\end{rem}
 
   \medskip
   
   \medskip
   \n\textit{Acknowledgement.}  Thanks are due  to  Marek Bo\.zejko, 
Simeng Wang 
and Mateusz Wasilewski   for useful communications.
Lastly I am  grateful to the referee for a very careful reading.

  \end{document}